\documentclass[a4paper,11pt]{article}
\usepackage[margin=2cm]{geometry}
\usepackage{titlesec}
\usepackage{multicol}
\usepackage{amsmath,amsfonts,amssymb,amsthm,mathrsfs}
\usepackage{hyperref}
\usepackage{helvet}
\usepackage{times}

\hypersetup{
    colorlinks=true,
    linkcolor=black,
    filecolor=blue,      
    urlcolor=blue,
    citecolor = black,
}
\usepackage[font=scriptsize,labelfont=it]{caption}

\usepackage[T1]{fontenc}

\usepackage{xcolor}
\usepackage{enumerate}
\usepackage{rotating}
\usepackage{pstricks, pst-node, pst-text, pst-3d}
\usepackage[color=blue!10]{todonotes}

\usepackage{wrapfig}

\usepackage{tikz}
\usepackage{blindtext}

\newtheoremstyle{exercise}
{}                % Space above
{}                % Space below
{\color{blue!70!black}}        % Theorem body font % (default is "\upshape")
{}                % Indent amount
{\color{blue!70!black} \bfseries}       % Theorem head font % (default is \mdseries)
{.}               % Punctuation after theorem head % default: no punctuation
{ }               % Space after theorem head
{}                % Theorem head spec

\newtheoremstyle{example}
{}                % Space above
{}                % Space below
{\color{olive!70!black}}        % Theorem body font % (default is "\upshape")
{}                % Indent amount
{\color{olive!70!black} \itshape}       % Theorem head font % (default is \mdseries)
{.}               % Punctuation after theorem head % default: no punctuation
{ }               % Space after theorem head
{}                % Theorem head spec

\newcounter{dummy} \numberwithin{dummy}{section}
 \numberwithin{dummy2}{section}
 \numberwithin{dummy3}{subsection}
\newtheorem{theorem}[dummy]{Theorem}

\newtheorem{lemma}[dummy]{Lemma}
\newtheorem{definition}[dummy]{Definition}
\newtheorem{proposition}[dummy]{Proposition}
\theoremstyle{example}
\newtheorem{example}[dummy]{Example}
\theoremstyle{remark}
\newtheorem{remark}[dummy]{Remark}
\definecolor{bleudefrance}{rgb}{0.19, 0.55, 0.91}

\colorlet{PaperBlue}{bleudefrance!70!black}

\titleformat*{\section}{\Large \bfseries \sffamily \color{PaperBlue}}
\titleformat*{\subsection}{\large \bfseries \sffamily}
\titleformat*{\subsubsection}{\bfseries \sffamily}
\titleformat*{\paragraph}{\bfseries}
\titleformat*{\subparagraph}{ \sffamily}

\numberwithin{equation}{section}

\newcommand{\calV}{\mathcal{V}}
\newcommand{\calH}{\mathcal{H}}
\newcommand{\scrF}{\mathscr{F}}
\newcommand{\calL}{\mathcal{L}}

\DeclareMathOperator{\GL}{GL}
\DeclareMathOperator{\Sp}{Sp}
\DeclareMathOperator{\SU}{SU}
\DeclareMathOperator{\Ort}{O}
\DeclareMathOperator{\Ad}{Ad}
\DeclareMathOperator{\dil}{dil}
\DeclareMathOperator{\Dil}{Dil}
\DeclareMathOperator{\symb}{\mathbf{symb}}
\DeclareMathOperator{\Symb}{\mathbf{Symb}}
\DeclareMathOperator{\free}{free}

\DeclareMathOperator{\so}{\mathfrak{so}}
\DeclareMathOperator{\se}{\mathfrak{se}}
\DeclareMathOperator{\su}{\mathfrak{su}}
\DeclareMathOperator{\E}{E}
\DeclareMathOperator{\End}{End}
\DeclareMathOperator{\Isom}{Isom}
\DeclareMathOperator{\isom}{\mathfrak{isom}}

\DeclareMathOperator{\Lbra}{[\! [}
\DeclareMathOperator{\Rbra}{]\! ]}

\DeclareMathOperator{\rank}{rank}

\DeclareMathOperator{\id}{id}
\DeclareMathOperator{\pr}{pr}
\DeclareMathOperator{\Ann}{Ann}
\DeclareMathOperator{\gl}{\mathfrak{gl}}
\DeclareMathOperator{\spn}{span}
\newcommand{\ve}{\varepsilon}
\newcommand{\Eucl}{\mathrm{Eucl}}

\newcommand{\Addresses}{{
{\it Erlend~Grong}, \textsc{University of Bergen, Department of Mathematics, P.O. Box 7803, 5020 Bergen, Norway}\par\nopagebreak
  \textit{E-mail:} \texttt{erlend.grong@uib.no}
}}

\providecommand{\keywords}[1]{\small \textit{Keywords---} #1}
\providecommand{\MSC}[1]{\small \textit{Mathematics Subject Classification (2020)---} #1}

\title{\color{PaperBlue} \huge \sf Curvature and the equivalence problem in sub-Riemannian geometry}
\author{\sf Erlend Grong\footnote{Author is supported by the grant GeoProCo from the Trond Mohn Foundation - Grant TMS2021STG02 (GeoProCo)}}
\date{\sf Srn\'i, January 15-22, 2022}
\begin{document}
\maketitle

\begin{abstract} These notes give an introduction to the equivalence problem of sub-Riemannian manifolds. We first introduce preliminaries in terms of connections, frame bundles and sub-Riemannian geometry. Then we arrive to the main aim of these notes, which is to give the description of the canonical grading and connection existing on sub-Riemann manifolds with constant symbol. These structures are exactly what is needed in order to determine if two manifolds are isometric. We give three concrete examples, which are Engel (2,3,4)-manifolds, contact manifolds and Cartan (2,3,5)-manifolds.

These notes are an edited version of a lecture series given at the \href{https://conference.math.muni.cz/srni/}{42nd Winter school: Geometry and Physics}, Snr\'i, Check Republic, mostly based on \cite{Gro20b} and other earlier work. However, the work on Engel (2,3,4)-manifolds is original research, and illustrate the important special case were our model has the minimal set of isometries.
\end{abstract}


\MSC{53C17,58A15}

\tableofcontents

%%%%%%%%%%%%%%%%%%%%%%%%%%

\section{Introduction} 

%%%%%%%%%%%%%%%%%%%%%%%%%%

{\it When are two spaces different?} Let us start with a simple object. The two-dimensional vector space $\mathbb{R}^2$ with its usual Euclidean metric. We give it the standard coordinates $(x,y)$ with corresponding vector field $\partial_x$, $\partial_y$. The Euclidean metric $g_{\Eucl}$ is given by
$$\langle \partial_x, \partial_x \rangle_{g_{\Eucl}} =1, \qquad \langle \partial_y, \partial_y \rangle_{g_{\Eucl}} =1, \qquad \langle \partial_x, \partial_y \rangle_{g_{\Eucl}} =0,$$
which can be written as either
$$g_{\Eucl} = \begin{pmatrix} 1 & 0 \\ 0 & 1 \end{pmatrix} \qquad \text{or} \qquad g_{\Eucl} = dx^2 + dy^2.$$
A general Riemannian metric $g = \langle \cdot  , \cdot \rangle_g$ on $\mathbb{R}^2$ is an inner product of tangent vectors that varies from point to point. It can be written as
$$g = \begin{pmatrix} g_{11} & g_{12} \\ g_{12} & g_{22} \end{pmatrix} \qquad \text{or} \qquad g = g_{11} dx^2 + 2 g_{21} dx dy+ g_{22}dy^2,$$
where $g_{ij}$ are smooth function with $g_{11} g_{22} - g_{12}^2 >0$ and $g_{11}>0$. This inner product can be considered as describing a shape as follows.

\begin{example} \label{ex:Metrics}
Let us consider the following subsets of $\mathbb{R}^3$ with the standard Euclidean metric.
\begin{multicols}{2}[]
\begin{enumerate}[\rm (i)]
\item Consider the subset
$$M= \left\{ (x, y, y^2) \, : \, (x,y) \in \mathbb{R}^2 \right\}.$$
We can view space this as $\mathbb{R}^2$ with the inner product
$$g = dx^2 + (1+ 4y^2) dy^2.$$
\end{enumerate}
\vfill\null
\columnbreak
\begin{center}
\includegraphics[height=4cm]{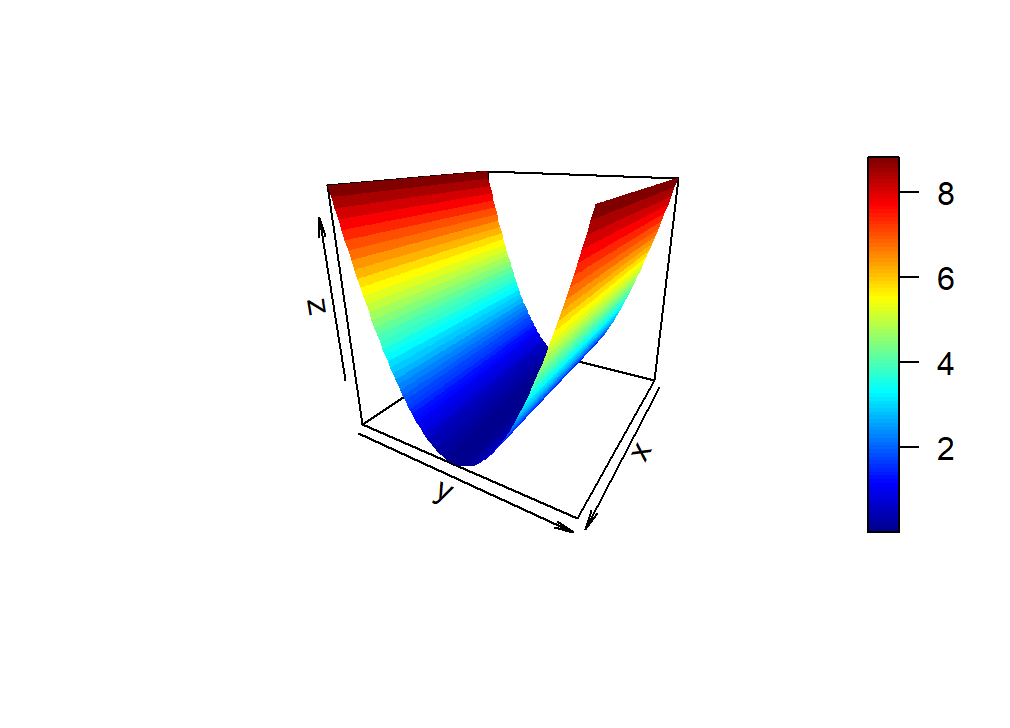}
\end{center}
\end{multicols}
\vspace{-1.2cm}

\begin{multicols}{2}[]
\begin{enumerate}[\rm (ii)]
\item Consider the subset
$$M= \left\{ (x, y, xy) \, : \, (x,y) \in \mathbb{R}^2 \right\}.$$
We can view this space as $\mathbb{R}^2$ with the inner product
$$g = (1+y^2) dx^2 + (1+x^2)dy^2 + 2 xy dx dy.$$
\end{enumerate}
\vfill\null
\columnbreak

\begin{center}
\includegraphics[height=4cm]{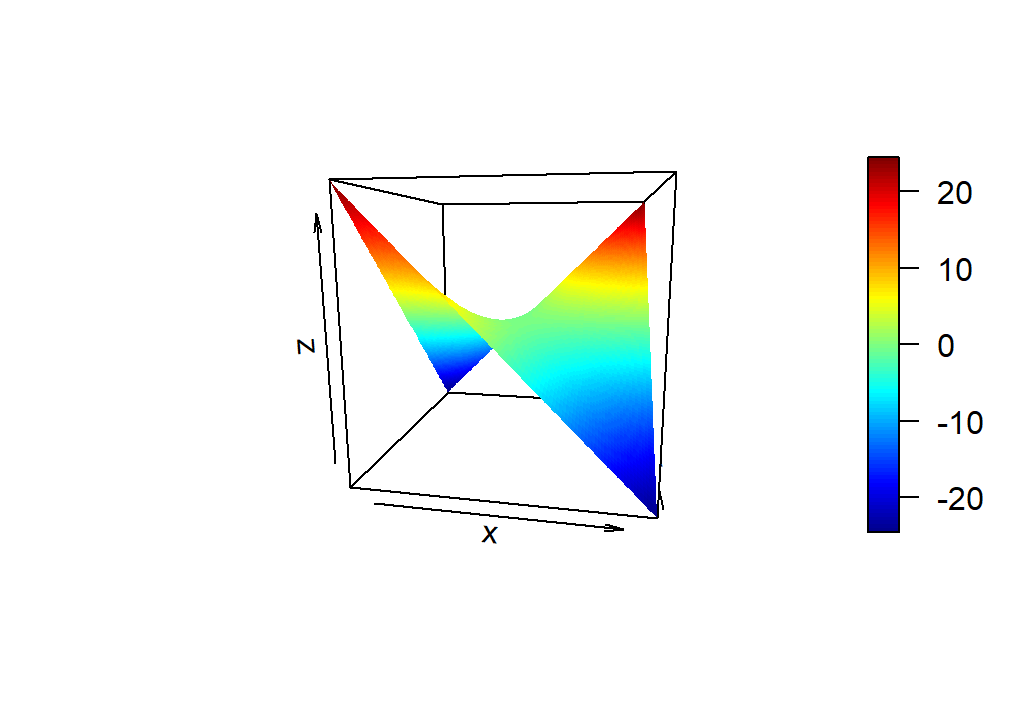}
\end{center}
\end{multicols}
\vspace{-1.2cm}

\begin{multicols}{2}[]
\begin{enumerate}[\rm (iii)]
\item We finally consider a punctured sphere $M = S^2\setminus \{(0,0,1)\}$.
We can see this as $\mathbb{R}^2$ using stereographic projection. Let $(p^0, p^1, p^2)$ be the coordinates of $\mathbb{R}^3$, and define
$$x = \frac{p^1}{1-p^0}, \qquad y = \frac{p^2}{1-p^0}.$$
The corresponding Riemannian metric is given by
$$g= \frac{4}{(1+x^2 +y^2)^2} (dx^2 + dy^2).$$
\end{enumerate}
\vfill\null
\columnbreak

\begin{center}
\includegraphics[height=5cm]{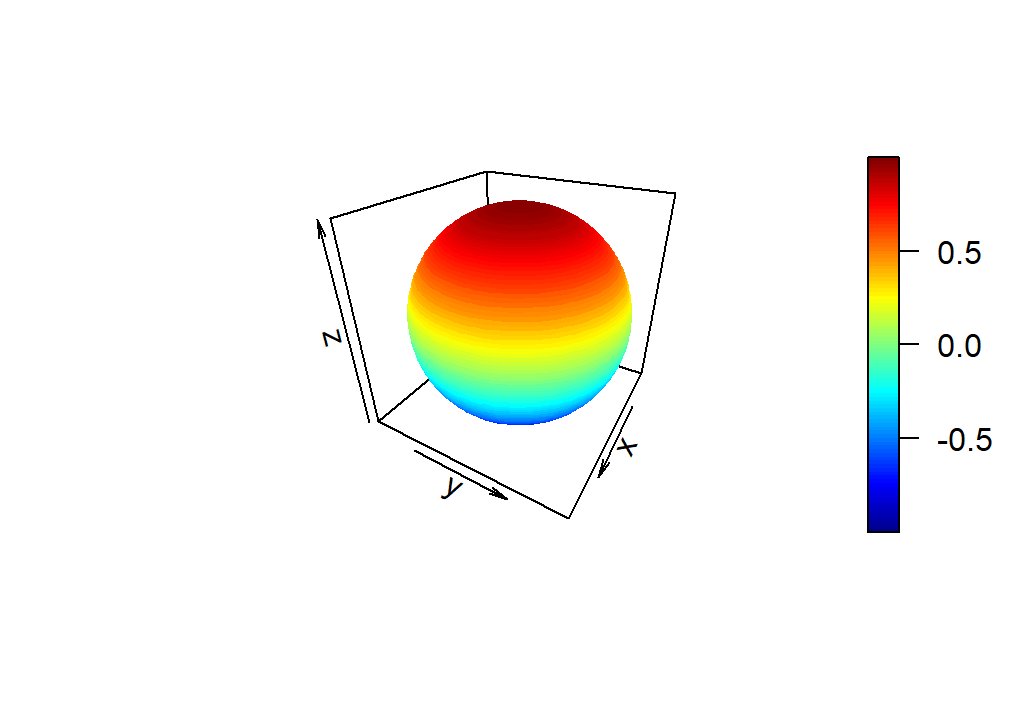}
\end{center}
\end{multicols}

\end{example}

The main question which this note will focus on, \emph{the equivalence problem}, is the following: How do we know if two spaces are the same? Roughly speaking, ``the same'' in the setting of Riemannian manifolds means that we can preform a change of variables such that one Riemannian metric transforms into the other, which will mean that all of the distances are preserved.

Let us look at the Example~\ref{ex:Metrics} and ask the simple question: Are any of these examples just a change of variables away from being $\mathbb{R}^2$ with the Euclidean metric? The answer is that this happens only in (i), with coordinate change $(u,v) = (x, \int_0^y \sqrt{1+t^2} dt)$. But how can we prove that it is impossible for the other examples, given that there are infinitely many coordinate changes? Or look at the following example:
$$g = (1+x^2)dx^2 + 2 \frac{x+y}{1+y^2} dx dy + \frac{1}{1+y^2}dy^2.$$
It is very difficult to find the change of variable to obtain the standard Euclidean space\footnote{The change of variable is $(u,v)= (x+\frac{1}{2} \log(1+y^2), \tan^{-1} y + \frac{1}{2} x^2)$} . Proving directly that a Riemannian metric cannot be rewritten as a flat metric is even more difficult. The answer is found in the invariant called Gaussian curvature and we can make a change of variable to get the Euclidean metric if and only if this invariant vanishes. That the Gaussian curvature is an invariant independent of choice of coordinates was observed by Carl Friedrich Gauss' Theorema Egregium in 1827. It was further generalized by his student Bernhard Riemann to more dimensions and shapes. The result in the end is the flatness theorem.
\begin{theorem}
A Riemannian manifold is locally isometric to the Euclidean space $\mathbb{R}^n$ if and only the curvature tensor of the Levi-Civita connection vanishes.
\end{theorem}
For the definition of the Levi-Civita connection, see Section~\ref{sec:Levi-Civita}. ``Locally isometric'' means that we can locally do a change of variable transforming the manifold to the Euclidean space with the standard Euclidean metric. We can make similar theorems for the sphere, hyperbolic space and so forward.

Our dealings with Riemannian manifolds will mainly serve to build intuition for the main topic: the equivalence problem of \emph{sub-Riemannian manifolds}. Sub-Riemannian geometry is a much younger subject\footnote{the term sub-Riemannian was first introduced in 1986 \cite{Str86}}, with a much less established structure. For a sub-Riemannian manifold, we have an inner product, but only in some directions, since the metric is only defined on a subbundle. Also in this setting, we can ask: \emph{when are two sub-Riemannian manifolds the same?} Consider two sub-Riemannian structures on $\mathbb{R}^3$, given by the following orthonormal bases
$$X_1 = \partial_x - \frac{1}{2} y \partial_z, \qquad X_2 = \partial_y + \frac{1}{2} x \partial_z,$$
for the first structure and
$$Y_1 = \left(\frac{\pi}{2} + \tan^{-1} z \right) \partial_x , \qquad Y_2= \frac{x}{(1+z^2)(\frac{\pi}{2} + \tan^{-1} z)} \partial_x + \frac{\frac{\pi}{2} + \tan^{-1} z}{x} \partial_y + \frac{1}{2} \frac{x}{\frac{\pi}{2} + \tan^{-1} z} \partial_z,$$
for the second structure. How would you know that these are just a change of variable from each other\footnote{the change of variable $(x,y,z) \mapsto (u,v,w)$ with $z = w$, $x = \frac{u \cos v}{\frac{\pi}{2} + \tan^{-1} w}$ and $u = \frac{u \sin v}{\frac{\pi}{2} + \tan^{-1} w}$}?
The equivalence problem for such manifolds is much more complicated. Through this notes, we aim to give some insight into challenges and results in this setting.

\paragraph{Outline} In Section~\ref{sec:Connections} we introduce the concept of affine connections and how such connections help us give a canonical basis for the frame bundle. In Section~\ref{sec:Riemannian}, we work with Riemannian metrics and the Levi-Civita connection. We will also prove the flatness theorem in Riemannian geometry, using the theory of Cartan connections. In Section~\ref{sec:subRiemannian} we define sub-Riemannian manifolds. We consider the flat sub-Riemannian spaces called Carnot groups, and define symbols at a point, which is a local nilpotent approximation of a sub-Riemannian manifold. Finally, we introduce sub-Riemannian frame bundle, and we use this construction to give a formula for a canonical choice of connection and grading for sub-Riemannian manifolds with constant symbol. We end with the following three examples. In Section~\ref{sec:Engel}, we do a complete computation of this connection and grading for the case of sub-Riemannian manifolds of growth vector $(2,3,4)$. In Sections~\ref{sec:Contact} and~\ref{sec:235} we give examples of flatness theorems for respectively contact manifolds and manifolds with growth vector (2,3,5), while leaving most of the details to \cite{Gro20b}.

\paragraph{Preliminaries, general references and further reading} These notes assume that the reader is familiar with the basic theory of manifolds, tangent bundles and differentials of maps between manifolds. The reader should also be familiar with Lie theory and principal bundles. Riemannian geometry and connections are introduced in the text, but it can be an advantage to have some pre-knowledge of these topics. We recommend the books \cite{Lee13,Lee97,Sha97} for more information, which can also be used as references for the information found in Section~\ref{sec:Connections} and Section~\ref{sec:Riemannian}. For references and further reading on sub-Riemannian manifolds, which is presented in Section~\ref{sec:subRiemannian}, we recommend books \cite{Mon02,ABB20}.

\paragraph{Formatting choices of the text} The original manuscript of the lecture notes contained several exercises. These have been turned into results (proposition, lemmas, etc) inside the text, reformulating them somewhat, and including the proofs.

\section{Connections} \label{sec:Connections}
\subsection{Manifolds, acceleration and connections}
On a general manifold with no further structure, there is really no good definition of a double derivative.
Consider a smooth curve $\gamma: [a,b] \to M$ into a manifold $M$. We can define the derivative $\dot \gamma:[a,b] \to TM$ as a section of the tangent bundle over $\gamma$. We note that $\dot \gamma(t) =0$ implies that $\gamma$ is constant. However, if we define $\frac{d}{dt} (\dot \gamma) = \ddot \gamma: [a,b] \to T(TM)$, then $\ddot \gamma = 0$ also implies $\gamma$ is constant. Hence, it is not really a good replacement for the second derivative in the Euclidean space. The problem can be considered as follows: If $X: M \to TM$ is a vector field, then $X_* : TM \to T(TM)$ determines change in both the `manifold part' and a `fiber part', but there is no canonical way of separating these. We could pick local coordinates, but this will not be a canonical choice, and it can sometimes be difficult to see which part of our expressions are coordinate dependent and which ones are not.

To get a proper definition of acceleration, we need the following additional structure. An affine connection on $TM$ is a map
$$\nabla: \Gamma(TM) \times \Gamma(TM) \to \Gamma(TM), \qquad (X,Y) \mapsto \nabla_X Y,$$
with the properties
\begin{enumerate}[\rm (I)]
\item Linearly property: the map $(X,Y) \mapsto \nabla_X Y$ is $C^\infty(M)$-linear in $X$ and $\mathbb{R}$-linear in $Y$.
\item Leibnitz property: If $f$ is a smooth function on $M$, then
$$\nabla_X(f Y) = (Xf) Y + f \nabla_X Y.$$
\end{enumerate}
We observe the following important consequences from the definitions of connections.

\begin{enumerate}[\rm (i)]
\item Since $\nabla_XY$ is tensorial in $X$, the vector $\nabla_{X}Y|_x$ only depends on $X|_x$. Hence, it makes sense to write $\nabla_v Y$ where $v$ is just an element in $TM$.
\item With a little more work, we can show that $\nabla_v Y$ only depends on the values of $Y$ along a curve tangent to~$v$. We can hence define covariant derivatives $\nabla_{\dot \gamma} Y(t)$ where $Y(t)$ is a vector field defined just along the differentiable curve $\gamma(t)$. We note that if $f(t)$ is a function depending on $t$, we have the Leibnitz property
$$\nabla_{\dot \gamma} f(t) Y(t) = \dot f(t) Y(t) + f(t) \nabla_{\dot \gamma} Y(t).$$
\item A vector field $Y(t)$ is called parallel along the curve $\gamma(t)$ if
$$\nabla_{\dot \gamma} Y(t) =0.$$
Parallel vector fields are the closest we can get to constant vector fields on a manifold. For any $v \in T_{\gamma(t_0)}M$, there exists a unique parallel vector field $Y(t)$ such that $Y(t_0) = v$. This allows us to define the parallel transport map
$$P_{t_0,t}: T_{\gamma(t_0)}M \to T_{\gamma(t)}M, \qquad P_{t_0,t}:v \mapsto Y(t), \qquad \text{$Y(t)$ parallel with $Y(t_0) =v$},$$
which is a linear isomorphism of vector spaces. In particular, parallel transport sends bases of the tangent space to bases of the tangent space.
\item Finally, we can take $Y(t) = \dot \gamma(t)$ in (ii), to define a 'double derivative' $\nabla_{\dot \gamma} \dot \gamma$. We define $\gamma$ to be \emph{a geodesic} if it satisfies
$$\nabla_{\dot \gamma} \dot \gamma =0.$$
Geodesics are analogues of constant speed curves in the sense that they are determined by an initial value and velocity.
\end{enumerate}

We consider the following tensors associated to failure of commutativity of the covariant derivatives. We first have the torsion tensor
$$T(X, Y) = \nabla_{X} Y- \nabla_{Y} X - [X,Y].$$
This can be interpreted as follows: Let $f:(a_1, b_1) \times (a_2, b_2)\to M$, $(s, t) \mapsto f(s,t)$ be a parametrized surface in $M$. Then
$$\nabla_{\partial_s} \partial_t f -\nabla_{\partial_t} \partial_s f = T(\partial_s f, \partial_t f).$$
Hence, derivatives in $s,t$ commute only if the torsion tensor vanishes. Next, we have the curvature
$$R(X, Y) Z = \nabla_X \nabla_Y Z - \nabla_Y \nabla_X Z - \nabla_{[X,Y]} Z.$$
It plays a similar role to torsion for vector fields, in the sense that if $Z(s,t)$ is a parametrized vector field for the surface $f(s,t)$, then
$$\nabla_{\partial_s f} \nabla_{\partial_t f} Z - \nabla_{\partial_t f} \nabla_{\partial_s f} Z= R(\partial_s f, \partial_t f)Z.$$

We finally note this important property for connections.
\begin{theorem} \label{th:Parallel}
Let $x \in M$ be any point. Then there is a local basis $X_1, \dots, X_n$ of vector fields around $x$ such that for any vector field $Y$, we have $\nabla_Y X_j |_x= 0$, $j=1, \dots, n$.
\end{theorem}
The main consequence of this result is that any information of the connection that is independent on a choice of basis does not depend on covariant derivatives of vector fields.

\begin{proof}
Let $y: U \to \mathbb{R}^n$ be a local chart such that the image of $U$ is a convex open set such that $y(x) = 0$. We will use this map to identify $U$ with $\tilde U = y(U)$ and so it is sufficient to prove the result for $\tilde U \subseteq \mathbb{R}^n$ and $x =0$. Let $(y_1, \dots, y_n)$ be the coordinates and define $\nabla_{\partial_{y_i}} \partial_{y_j} = \sum_{k=1}^n \Gamma_{ij}^k \partial_{y_k}$. For any $y \in \tilde U$, define $\gamma_y(t) = ty$. We then define $X_1|_y$, $\dots$, $X_n|_y$ by parallel transport along $\gamma_y(t)$. In other words, if $Y(y,t) = \sum_{k=1}^n Y_k(y,t) \partial_{y_k}$ is the solution of
$$\frac{\partial}{\partial t} Y_k + \sum_{i_1,i_2=1}^n y_{i_1} Y_{i_2} \Gamma_{i_1i_2}^k(ty) = 0,\qquad Y(y,0) = X_k|_0$$
then $X_j|_y =Y(y,1)$ which is a smooth function in $y$. Finally, by definition each $X_j$ is parallel at any direction at $0$. The result follows.
\end{proof}

\begin{remark} Rather than just restricting ourselves to the tangent bundle and vector fields, we could have defined an affine connection on a general vector bundle $E$, where now in the expression $\nabla_X Y$, $X$ is a section of $TM$ while $Y$ is a section of $E$. All of the above properties holds, except that there is no definition of the torsion tensor or geodesics. Observe from the proof of Theorem~\ref{th:Parallel}, that it can be modified to work on any vector bundle choosing an initial basis of the vector bundle at the initial point. Also, if some property of a basis is preserved under parallel transport (such as being orthonormal, symplectic, etc), then we may assume that the local basis in Theorem~\ref{th:Parallel} has this property.
\end{remark}

\subsection{Frame bundles}
Let $M$ be a general differential manifold of dimension $n$ and let $\mathbb{R}^n$ be the Euclidean space. Even if $M$ happens to be diffeomorphic to $\mathbb{R}^n$, there is a structure that exists on $\mathbb{R}^n$ that we do not have on $M$. Namely, $\mathbb{R}^n$ has a canonical basis of vector fields spanning the tangent space, in this case given by the derivatives corresponding to the standard coordinates. The frame bundle is a smart construction that steals this property from $\mathbb{R}^n$ to $M$, but at the expense that we have to work on a fiber bundle above $M$ instead on $M$ directly. We also need a connection $\nabla$ on $M$ to have the full canonical basis.

Let $e_1, \dots, e_n$ be the standard basis of $\mathbb{R}^n$. \emph{A frame} at $x \in M$ is a choice of basis $u_1$, $\dots$, $u_n$ for $T_x M$. Equivalently, we can consider a frame as an invertible linear mapping $u: \mathbb{R}^n \to T_x M$. The correspondence is given by
\begin{equation} \label{maptoframe} u_j = u(e_j).\end{equation}
We write the set of all such frames as $\GL_x(M)$.

Let $\GL(n)$ be the general linear group of real invertible $n \times n$ matrices. If we have a frame $u: \mathbb{R}^n \to T_x M$, then for any $a \in \GL(n)$, we can define a new frame $u \cdot a$ by $u \cdot a= u \circ a: \mathbb{R}^n \to T_x M$ by precomposition. In other words, if $U = u \cdot a$, then
$$U_j = \sum_{i=1}^n a_{ij} u_i.$$
\begin{lemma}
For any pair of elements $u, U \in \GL_x(M)$, there is a unique element $a \in \GL(n)$ such that $U = u \cdot a$.
\end{lemma}
\begin{proof}
Since $u_1, \dots, u_n$ and $U_1,\dots, U_n$ are bases, every vector should have a unique decomposition in these bases. Hence, there must exist unique coefficients $(a_{ij})$ and $(A_{ij})$ such that we can write
$$U_j = \sum_{i=1}^n a_{ij} u_i, \qquad u_j = \sum_{i=1}^n A_{ij} U_i.$$
Hence
$$u_k = \sum_{i,j=1}^n A_{ik}  a_{ji} u_j,$$
showing that $(a_{ij})$ is invertible with inverse $(A_{ij})$. If we write $a = (a_{ij}) \in \GL(n)$, it follows that $U = u \cdot a$, which is unique by construction.
\end{proof}

We use the above action to construct a principal bundle called the frame bundle
$$\GL(n) \to \GL(M) \stackrel{\pi}{\to} M,$$
with the fiber over $x \in M$ being $\GL_x(M)$. On the frame bundle, we have the following structures.
\begin{enumerate}[\rm (i)]
\item \emph{The vertical bundle:} The frame bundle $\GL(M)$ is an $n+n^2$-dimensional manifold, and considering differential $\pi_*$ of $\pi: \GL(M) \to M$, and taking its kernel, we obtain the vertical bundle $\calV = \ker \pi_*$ of rank $n^2$. Such a bundle exists for all fiber bundles.
\item \emph{Canonical vertical vector fields:} If we take a frame $u$ at $\GL_x(M)$, we can rotate it at $x \in M$ without moving on the manifold. Derivatives of such rotations will hence be in $\calV$.
Let $\gl(n)$ be the the set of all real $n\times n$-matrices, the Lie algebra of $\GL(n)$. If $A \in \mathfrak{gl}(n)$, we define
$$\xi_A|_u = \left. \frac{d}{dt} u \cdot e^{At} \right|_{t=0} \in \calV_u.$$
This formula gives us a globally defined vector field $\xi_A$ on $\GL(M)$. The map $A \mapsto \xi_A|_u$ has a trivial kernel and since $\gl(n)$ and $\calV_u$ have the same dimension,
$$\calV_u = \left\{ \xi_A |_u \, : \, A \in \gl(n) \right\}, \qquad u \in \GL(M).$$
Hence, $\calV$ is a trivializable bundle. Such vector fields exists on all principal bundles.
\item \emph{The tautological one-form:} We define an $\mathbb{R}^n$-valued one-form $\theta = (\theta_1, \dots, \theta_n)$ by
$$\theta(w) = u^{-1} \pi_* w, \qquad w \in T_u \GL(M).$$
In other words, we take $w$ living in $T_u \GL(M)$, use $\pi_*$ to send it to $T_{\pi(u)} M$, then $\theta(w)$ is the result of writing $\pi_* w$ in the frame $u$. Observe that $\ker \theta = \calV$.
The form $\theta$ is very special for the frame bundle of the tangent bundle.
\end{enumerate}
Since we have an action of $\GL(n)$ on $\GL(M)$, we can also consider the induced action of $\GL(n)$ on $T\GL(M)$. To define this, write $r_a(u) = u \cdot a$, $u \in \GL(M)$, $a \in \GL(n)$. For any vector $w \in T\GL(M)$, define
$$w \cdot a := (r_a)_* w.$$

\begin{lemma}
For any $A \in \gl(n)$, $a \in \GL(n)$,
$$\xi_A \cdot a = \xi_{\Ad(a^{-1}) A}.$$ 
\end{lemma}

\begin{proof}
For any frame $u \in \GL(n)$, we observe
\begin{align*}
& \xi_A |_u \cdot a = \frac{d}{dt} u \cdot \exp(At) \cdot a |_{t=0} = \frac{d}{dt} u \cdot a \cdot (a^{-1} \exp(At) \cdot a) |_{t=0} \\
& = \frac{d}{dt} u \cdot a \cdot \exp(\Ad(a^{-1})At) |_{t=0} = \xi_{\Ad(a^{-1}) A} |_{u \cdot a},
\end{align*}
giving us the desired result.
\end{proof}

\subsection{Connections on frame bundles} \label{sec:ConFrame}
An Ehresmann connection $\calH$ on $\pi: \GL(M) \to M$ is a choice of complement to $\calV$. In other words,
$$T\GL(M) = \calH \oplus \calV.$$
We can interpret this complement at follows. If $u \in \GL_x(M)$, then $\pi_{*,u}: T_u \GL(M) \to T_x M$ has kernel $\calV_u$. If we choose a complement $\calH_u$, then $\pi_{*,u}|_{\calH_u}$ is invertible and we can define an inverse $h_u$
so that $h_u v$ is the unique element in $\calH_u$ satisfying $\pi_* h_u v = v$. An Ehresmann connection is called principal if it is invariant under the action $\GL(n)$. That is
$$\calH_{u} \cdot a = \calH_{u \cdot a}.$$
Equivalently, $h_u v \cdot a = h_{u \cdot a} v$ for any $u \in \GL_x(M)$, $v\in T_x M$, $x \in M$, $a \in \GL(n)$. We can also describe principal connections in the following way.

\begin{lemma} Given a principal connection $\calH$ with corresponding horizontal lifts, define a one-form $\omega: T\GL(n) \to \gl(n)$ such that
$$\omega(h_u v) =0, \qquad \omega(\xi_A) = A.$$
Then $\omega(w \cdot a) = \Ad(a^{-1}) \omega(w)$, $w \in T\GL(M)$, $a \in \GL(n)$.
\end{lemma}
The form $\omega$ is called \emph{the connection form} of $\calH$.
\begin{proof}
Write an arbitrary vector $w \in T_u \GL(n)$ as $w = h_u v + \xi_{A}|_u$, where $\pi_* w = v$ and $\omega(v) = A$.
Then by the previous result,
$$w \cdot a = h_{u \cdot a} v + \xi|_{\Ad(a^{-1}) A} |_{u \cdot a},$$
and so
$$\omega(w \cdot a) = \Ad(a^{-1})A = \Ad(a^{-1}) \omega(w).$$
Since $w$ was arbitrary, the result follows.
\end{proof}

We can construct a principal connection from an affine connection. Let $U(t)$ be a curve in $\GL(M)$ such that $\pi(U(t)) = \gamma(t)$. Assume that $\dot \gamma(0) = v$, that $U(0)= u$ and that each $U_j(t)$ is parallel along $\gamma(t)$. We can then define
$$h_{u} v = \dot U(0),$$
so the derivative of a parallel frame moving in the direction of $v$. We then define
$$\calH_x = \{h_{u} v \, :\, u \in  \GL_x(M), v \in T_x M\}.$$

\begin{lemma}
$\calH$ is a principal Ehresmann connection.
\end{lemma}

\begin{proof}
If $U(t)$ is a parallel frame and since linear combinations of parallel vector fields are still parallel, it follows that $U(t) \cdot a$ is also a parallel frame. If $U(0) = u$, then
$$h_{u} v \cdot a = \frac{d}{dt} (U(t) \cdot a)|_{t=0} = h_{u \cdot a} v.$$
It follows that $\calH$ is a principal Ehresmann connection.
\end{proof}

Now something interesting happens: We get a canonical basis of~$\calH$. For any $j =1, 2, \dots, n$, define
$$H_j |_u = h_u u_j.$$
For any element in $p \in \mathbb{R}^n$, we can define $H_p = \sum_{j=1}^n p_j H_j$.
\begin{lemma}
We have bracket relations for $p,q \in \mathbb{R}^n$, $A, B \in \gl(n)$,
\begin{equation}
\label{frame_brackets} [H_p, H_q] = - \xi_{\bar{R}(p,q)} -H_{\bar{T}(p,q)}, \qquad [\xi_A, H_q] = H_{Aq}, \qquad [\xi_A, \xi_B] = \xi_{[A,B]},\end{equation}
where
$$\bar{T}(p,q)|_u = u^{-1} T(u(p), u(q)), \qquad \bar{R}(p,q)|_u = u^{-1} R(u(p), u(q)) u.$$
\end{lemma}
Before the proof, we introduce \emph{the Hessian} $\nabla^2$ of a connection $\nabla$, from
$$\nabla^{2}_{X,Y} = \nabla_X \nabla_Y - \nabla_{\nabla_X Y}, \qquad X, Y \in \Gamma(TM).$$
Observe that $\nabla_{X,Y}^2$ is tensorial in both $X$ and $Y$. Furthermore,
$$\nabla_{X,Y}^2 - \nabla_{Y,X}^2 = R(X,Y) - \nabla_{T(X,Y)}.$$
Finally observe that if $X,Y, Z$ are vector fields and $\alpha$ is a one-form, then
$$X (\nabla_Y \alpha)(Z) = (\nabla_{X,Y}^2 \alpha)(Z) + (\nabla_{\nabla_X Y}\alpha)(Z) + (\nabla_Y \alpha)(\nabla_X Z).$$
\begin{proof}
We prove the equation \eqref{frame_brackets} in the following steps. For any one-from $\alpha$ and $r \in \mathbb{R}^n$, introduce function $F_{\alpha,r}: \GL(M) \to\mathbb{R}$ by
$$F_{\alpha,r}(u) = \alpha(u(r)).$$
For any $u \in \GL(M)$, and $p \in \mathbb{R}^n$, let $\gamma(t) = \gamma^{u,p}(t)$ be a curve with $\dot \gamma(0) = u(p)$ and let $U(t) = U^{u,p}(t)$ be the result of parallel transport of $u(p)$ along $\gamma(t)$ and we have
\begin{align*}
H_q F_{\alpha,r}(u) & = \frac{d}{dt} \alpha(U^{u,q}(t)(r)) |_{t=0} = (\nabla_{\dot \gamma^{u,q}(0)} \alpha)(u(r)) = (\nabla_{u(q)} \alpha)(u(r)), \\
\xi_A F_{\alpha,r}(u) &= \frac{d}{dt} \alpha( u(\exp(At)r)) |_{t=0} =\alpha( u(Ar)) = F_{\alpha,Ar}(u), \\
H_p H_q F_{\alpha,r}(u) &= \frac{d}{dt} (\nabla_{U^{u,p}(t)(q)} \alpha)(U^{u,p}(t)(r))|_{t=0} = (\nabla^2_{u(p),u(q)} \alpha)(u(r)), \\
\xi_A H_q F_{\alpha,r}(u) & = \frac{d}{dt} (\nabla_{u(\exp(At)q)} \alpha)(u(\exp(At)r)) |_{t=0} = (\nabla_{u(Aq)} \alpha)(u(r)) + (\nabla_{u(q)} \alpha)(u(Ar)), \\
H_p \xi_A F_{\alpha,r}(u) &= H_p F_{\alpha,Ar}(u) = (\nabla_{u(p)} \alpha)(u(Ar)), \\
\xi_B \xi_A F_{\alpha,r}(u) &= \xi_B F_{\alpha,Ar}(u) = F_{\alpha,BAr}(u).
\end{align*}

Form the above equations, we see that if $w \in T\GL(M)$ and if $(F_{\alpha,r})_* w = 0$ for any $\alpha \in \Gamma(T^*M)$ and $r \in \mathbb{R}^n$, then $w = 0$. Hence, we can use the functions $F_{\alpha,r}$ to determine vectors. Furthermore, using the above relations,
\begin{align*}
[H_p, H_p] F_{\alpha, r}(u) & = (\nabla^2_{u(p),u(q)} \alpha)(u(r)) - (\nabla^2_{u(q),u(p)} \alpha)(u(r)) \\
& = (R(u(p),u(q)) - \nabla_{T(u(p),u(q))} \alpha)(u(r)) \\
& = - \alpha(u\bar{R}(p,q)|_ur )- (\nabla_{u \bar{T}(p,q)|_u} \alpha)(u(r)) \\
& = - \xi_{\bar{R}(p,q)|_u} F_{\alpha,r}(u) - H_{\bar{T}(p,q)|_u} F_{\alpha,r}(u),
\end{align*}
\begin{align*}
[\xi_A, H_p] F_{\alpha, r}(u) & = (\nabla_{u(Aq)} \alpha)(u(r)) = H_{Aq} F_{\alpha,r}(u), \\
[\xi_A, \xi_B] F_{\alpha,r} & = F_{\alpha, [A.B]r} = \xi_{[A,B]}F_{\alpha,r}.
\end{align*}
The result follows.
\end{proof}

\section{Riemannian manifolds and curvature} \label{sec:Riemannian}
\subsection{Riemannian manifolds and compatible connections}
A Riemannian manifold $(M,g)$ is a connected manifold $M$ and a smoothly varying inner product
$g = \langle \cdot, \cdot \rangle_g$ on the tangent space. For any smooth curve $\gamma: [a,b] \to M$, define its length by
$$L(\gamma) = \int_a^b \langle \dot \gamma(t), \dot \gamma(t) \rangle_g^{1/2} \, dt.$$
Define the Riemannian distance $d_g(x,y)$ as the infimum of lengths of curves connecting $x$ and $y$. \emph{An isometry} $\varphi$ between two Riemannian manifolds $(M,g)$ and $(\tilde M, \tilde g)$ is a diffeomorphism such that
$$\langle \varphi_* v, \varphi_* w \rangle_{\tilde g} = \langle v,  w \rangle_{g}.$$

Let $\nabla$ be an affine connection on $TM$. We say that $\nabla$ is compatible with the Riemannian metric $g$ if
$$X\langle Y_1, Y_2 \rangle_g = \langle \nabla_X Y_1, Y_2 \rangle_g + \langle Y_1 , \nabla_X Y_2 \rangle_g.$$
We have the following alternative description.
\begin{lemma} The following are equivalent.
\begin{enumerate}[\rm (i)]
\item The connection $\nabla$ is compatible with the Riemannian metric $g$.
\item For any point $x \in M$, any orthonormal frame $u_1, \dots, u_n$ of $T_xM$ and any smooth curve $\gamma:[0,T] \to M$ with $\gamma(0) =x$, if $U_1(t), \dots, U_n(t)$ is the result of parallel transporting the frame along $\gamma$, then this frame is also orthonormal for any $t$.
\end{enumerate}
\end{lemma}

\begin{proof}
(i) $\Rightarrow$ (ii): We see that if $U_1(t), \dots, U_n(t)$ is a parallel basis along a curve, then
$$\frac{d}{dt} \langle u_i, u_j \rangle_g = \langle \nabla_{\dot \gamma} U_i(t), U_j(t) \rangle_g + \langle U_i(t) , \nabla_{\dot \gamma} U_j(t) \rangle_g =0 .$$
Hence, if the frame was orthonormal at $t=0$, it will remain orthonormal. \\
(ii) $\Rightarrow$ (i): Let $X, Y_1, Y_2$ be any vector field and $x \in M$ any point. Let $\gamma:(-\ve,\ve) \to M$ be a curve such that $\gamma(0) = x$ and such that $\dot \gamma(0) = X|_x$. Let $U_1(t), \dots, U_n(t)$ be a parallel orthonormal frame along $\gamma(t)$, and write
$$Y_j|_{\dot \gamma(t)} = \sum_{i=1}^n Y_j^i(t) U_i(t).$$
Then
$$\nabla_{\dot \gamma(t)} Y_j = \sum_{i=1}^n \dot Y_j^i(t) U_i(t).$$
Hence, we have that
\begin{align*}
X\langle Y_1, Y_2 \rangle_x(x) & = \frac{d}{dt} \langle Y_1|_{\gamma(t)}, Y_2 |_{\gamma(t)} \rangle_g |_{t=0} \stackrel{\text{$u$ orthonormal}}{=}  \frac{d}{dt} \sum_{i=1}^n Y_1^i(t) Y_2^i(t) |_{t=0} \\
& = \sum_{i=1}^n \dot Y_1^i(0) Y_2^i(0) +  \sum_{i=1}^n Y_1^i(0) \dot Y_2^i(0) \\
& = \langle \nabla_{\dot \gamma(0)} Y_1, Y_2 \rangle_g + \langle Y_1, \nabla_{\dot \gamma(0)} Y_2 \rangle_g  = \langle \nabla_X Y_1, Y_2 \rangle_g(x) + \langle Y_1, \nabla_{X} Y_2 \rangle_g(x).
\end{align*}
Since the vector fields and point were arbitrary, the result follows.
\end{proof}

We will use the above result in the following way. For any $x \in M$, we define $\Ort_x(M)$ as the space of linear \underline{isometries} from the Euclidean space $\mathbb{R}^n$ to $T_x M$. It can also be considered as the space of all orthonormal frames of $T_xM$ using the correspondence in \eqref{maptoframe}. We then have a transitive action of the orthogonal group $\Ort(n)$ on $\Ort_x(M)$ by precomposition, $u \cdot a = u \circ a$, $u \in \Ort_x(M)$, $a \in \Ort(n)$.
Recall that $\Ort(n)$ is the group of $n \times n$ real matrices satisfying $a^T a = 1$. It has Lie algebra $\so(n)$ of skew-symmetric matrices satisfying $A^T + A =0$.

We can use this action to build a principal bundle
$$\Ort(n) \to \Ort(M) \stackrel{\pi}{\to} M,$$
called \emph{the orthonormal frame bundle}. We still have the tautological one-form $\theta$ on this bundle with the same definition, and we have the vector fields $\xi_A$ which are only defined for $A \in \so(n)$. Since the compatible connection $\nabla$ preserves orthonormal frames under parallel transport, we can create the principal connection $\calH$ with corresponding connection form $\omega$ and canonical horizontal vector fields $H_p$, $p \in \mathbb{R}^n$ in the same way we did in Section~\ref{sec:ConFrame}.

\subsection{Integrability of Lie algebra valued forms} \label{sec:Integrate}
First a definition for Lie algebra-valued forms. Let $\alpha$ be a real valued $k$-form on $M$ and let $A$ be an element in the Lie algebra. Then $\alpha \otimes A$ is a $k$-form with values in $\mathfrak{g}$, that is
$$(\alpha \otimes A)(v_1, \dots, v_k) = (\alpha(v_1, \dots, v_k)) A.$$
Any $\mathfrak{g}$-valued $k$-form can be written as a sum $\eta = \sum_{j=1}^k \alpha^j \otimes A_j$. 
For such forms $\eta =  \sum_{j=1}^k \alpha^j \otimes A_j$ and $\psi = \sum_{j=1}^k \beta^j \otimes B_j$, we can define
$$[\eta , \psi] = \sum_{i,j=1}^k (\alpha^i \wedge \beta^j) \otimes [A_i, B_j].$$
Observe the following properties.
\begin{enumerate}[\rm (i)]
\item $[\eta , \psi] = (-1)^{deg(\psi)deg(\eta)+1} [\psi , \eta]$,
\item If $\eta$ is a one-form, then $[\eta,\eta](v,w) = 2[\eta(v),\eta(w)]$.
\end{enumerate}

\begin{example}
Let $G$ be any Lie group with Lie algebra $\mathfrak{g}$. The left Maurer-Cartan form $\eta$ on $G$, is a $\frak{g}$-valued one-form given by
$$\eta(v) = a^{-1} \cdot v, \qquad v \in T_a G.$$
We can check that this one-form then satisfies,
\begin{equation} \label{leftMC}
d\eta + \frac{1}{2} [\eta, \eta] =0.
\end{equation}
To see that this result holds, we only need to check that it holds for left invariant vector fields. If $A,B$ are elements of the Lie algebra and $\hat A$ is the vector field defined by left translation of $A$, then
\begin{align*}
d\eta(\hat A, \hat B) & = \hat A\eta(\hat B) - \hat B\eta(\hat A) - \eta([\hat A,\hat B]) =- [A,B] = - \frac{1}{2} [\eta(\hat A), \eta(\hat B)].
\end{align*}
since $\eta(\hat A)$ and $\eta(\hat B)$ are constant.

Next, let $f: M \to G$ be any smooth map and define $\psi = f^* \eta$. Then by the properties of the pull-back, we have that $\psi$ is a $\mathfrak{g}$-valued one-form also satisfies $d\psi+ \frac{1}{2} [\psi, \psi] =0$.
\end{example}
We now have the following important result. Recall that a one-form $\alpha$ can locally be seen as the differential of a function if and only if $d\alpha = 0$. We have a similar result for Lie algebra-valued forms.
\begin{theorem} \cite{Sha97}
Let $G$ be a Lie group with Lie algebra $\mathfrak{g}$ an let $\eta$ be its Maurer-Cartan form. If $\psi$ is a $\mathfrak{g}$-valued one-form on a manifold $M$, then there is locally a function $f$ from $M$ to $G$ such that $\psi = f^* \eta$ if and only if
$$d\psi + \frac{1}{2} [\psi, \psi] = 0.$$
\end{theorem}

Let us now return to the orthonormal frame bundle $\Ort(M)$. Recall the brackets in \eqref{frame_brackets} for the vector fields $H_p$, $\xi_A$, $p \in \mathbb{R}^n$, $A \in \so(n)$. Notice that if $\theta$ is the tautological one-from and $\omega$ is the connection form, we observe the relations
\begin{align*}
\theta(H_p) &= p, & \theta(\xi_A) & = 0, \\
\omega(H_p) & = 0, & \omega(\xi_A) & = A.
\end{align*}
Using these relations, we have the following.
\begin{proposition} 
The equations \eqref{frame_brackets} can be rewritten as
\begin{equation} \label{Cartan_eq}
d\theta + [\omega, \theta] =\Theta, \qquad d\omega + \frac{1}{2} [\omega, \omega] = \Omega,
\end{equation}
where
$$\Theta(\xi_A, \cdot ) = 0, \qquad \Omega(\xi_A, \cdot ) = 0, \qquad \Theta(H_p, H_q) = \bar{T}(p,q), \qquad \Omega(H_p, H_q) = \bar{R}(p,q).$$
\end{proposition}
\begin{proof}
We only need to show that the above formulas holds for the vector fields $H_p$, $p \in \mathbb{R}^n$ and $\xi_A$, $A \in \so(n)$. Note that $d\theta(H_p, H_q) =- \theta([H_p, H_q])$, since $\theta(H_p)$ is a constant. 
We then see that
\begin{align*}
& d\theta(H_p, H_q) + [\omega, \theta](H_p, H_q) = d\theta(H_p, H_q) = - \theta([H_p, H_q]) \\
& = \theta(\xi_{\bar{R}(p,q)} + H_{\bar{T}(p,q)}) = \bar{T}(p,q) = \Theta(H_q, H_q).
\end{align*}
The other calculations are similar.
\end{proof}

We can rewrite \eqref{Cartan_eq} in $\se(n)$, the Lie algebra of the group of Euclidean transformations $\E(n)$. The group $\E(n)$ can be written as matrices,
$$\begin{pmatrix} a & p \\ 0 & 1
\end{pmatrix}, \qquad a\in \Ort(n), p \in \mathbb{R}^n,$$
but we can equivalently consider this group as the space $\Ort(n) \times \mathbb{R}^n$ with group operation
$$(a, p) \cdot (b, q) = (ab, p+aq), \qquad a,b \in \Ort(n), p,q \in \mathbb{R}^n.$$
The corresponding Lie algebra can be considered as the space $\so(n)\times \mathbb{R}^n$ with brackets,
$$[(A,p), (B,q)] = ([A,B], Aq-Bp), \qquad A, B \in \so(n), p,q \in \mathbb{R}^n.$$ 
We can then write $\psi = (\omega, \theta)$ and the equations \eqref{Cartan_eq} now become
$$d\psi + \frac{1}{2} [ \psi, \psi] = (\Omega, \Theta).$$
We see that the curvature and the torsion is exactly the obstruction to $\psi$ being able to be integrated, at least locally.

\subsection{The Levi-Civita connection and the flatness theorem revisited} \label{sec:Levi-Civita}
For a Riemannian manifold $(M,g)$, \emph{the Levi-Civita connection} is the unique connection that is compatible with~$g$ and with torsion $T = 0$.

We mentioned in the introduction that we can determine if a manifold is flat using the Levi-Civita connection. We will do a proof of the flatness theorem in Riemannian geometry using frame bundles.
 \begin{theorem}[Flatness theorem]
 A Riemannian manifold $(M,g)$ is locally isometric to the Euclidean space if and only if the curvature of the Levi-Civita connection vanishes.
\end{theorem}
\begin{proof} $\Rightarrow$: We can verify directly that the curvature of the Levi-Civita connection of the Euclidean space vanishes, and the curvature is a local invariant. \\
 
\noindent $\Leftarrow$: Let $\eta$ be the Maurer-Cartan form on $\mathfrak{se}(n)$. Assume that the curvature of the Levi-Civita connection vanishes. Then the corresponding form $\psi: \Ort(M) \to \se(n)$ satisfies $\psi + \frac{1}{2} [\psi, \psi] = 0$. This means that locally for some open set $V \subseteq \Ort(M)$, we have a function $f: V \to \E(n)$ such that $\psi = f^* \eta$. Since $\psi$ is a linear isomorphism at every point, $f$ is a local diffeomorphism. Since $\psi(v \cdot a) = \Ad(a^{-1}) \psi(v)$ for $a \in \Ort(n)$, we have
 $$f(x \cdot a) = f(x) \cdot a,$$
 meaning that $f$ descends to a map from $V/\Ort(n)$ to $\mathbb{R}^n$. Finally, since the map $v \mapsto h_u v \mapsto \theta(h_u v)$ is a linear isometry on every point, the decended map is an isometry on every tangent space.
 \end{proof}
 
 Observe that in this proof, we needed two steps:
\begin{itemize}
\item Normalize the connection to remove the torsion.
\item Then we can set up the condition of the curvature vanishing afterwards.
\end{itemize}

\subsection{Cartan connections}
The form $\psi$ is a special case of what is known as a Cartan connection. Consider first a manifold $M$ of dimension~$n$. Let $\mathfrak{g}$ be a Lie algebra with a subalgebra $\mathfrak{h}$ of codimension $n$. Let $H$ be a Lie group with Lie algebra $\mathfrak{h}$ and let $\Ad$ be a representation of $H$ on $\mathfrak{g}$ extending the usual adjoint action of $H$ on $\mathfrak{h}$. Assume that we have a principal bundle
$$H \to P \stackrel{\pi}{\to} M.$$
\emph{A Cartan connection $\psi$ on $P$ modeled on $(\mathfrak{g},\mathfrak{h})$} is a $\mathfrak{g}$-valued one form $\psi: TP \to \mathfrak{g},$
such that
\begin{enumerate}[\rm (i)]
\item For each $p \in P$, $\psi|_p$ is a linear isomorphism from $T_pP$ to $\mathfrak{g}$.
\item For each $a \in H$, $v \in TP$,
$$\psi(v \cdot a) = \Ad(a^{-1}) \psi(v).$$
\item For every $D \in \mathfrak{h}$, $p \in P$, we have $\psi(\frac{d}{dt} p \cdot  \exp_H(tD) |_{t=0}) = D$.
\end{enumerate}
The curvature of a Cartan connection is the obstruction to integrability in the sense of Section~\ref{sec:Integrate}. It is represented by a smooth function $\kappa: P \to \wedge^2 (\mathfrak{g}/\mathfrak{h})^* \otimes \mathfrak{g}$ such that,
$$\kappa(\psi(\cdot), \psi(\cdot))  = d\psi + \frac{1}{2} [\psi, \psi].$$
It follows from the definition of Cartan connections that if $\psi(w)$ takes values in $\mathfrak{h}$, then $w$ inserted into $d\psi + \frac{1}{2}[\psi,\psi]$ will vanish. Hence, we write $\wedge^2 (\mathfrak{g}/\mathfrak{h})^* \otimes \mathfrak{g}$ as the codomain of the curvature.

\section{Sub-Riemannian geometry} \label{sec:subRiemannian}
\subsection{Bracket-generating subbundle}
Let $M$ be a connected manifold. Let $E$ be a subbundle of the tangent bundle $TM$. A subbundle is called \emph{bracket-generating} if the following property holds:
Let $\Gamma(E)$ be sections of $E$, that is, all vector fields that take values in $E$. We consider such vector fields and all possible brackets
$$\spn\{ X_1, [X_1, X_2], [X_1, [X_2, X_3]] , \dots,\}, \qquad X_j \in \Gamma(E).$$
If the above vector fields span the entire tangent bundle, we say that $E$ is bracket-generating. Relative to $E$, we say that an absolutely continuous curve $\gamma: [a,b] \to M$
is \emph{horizontal} if $\dot \gamma(t) \in E_{\gamma(t)}$ for almost every $t$. The bracket-generating then have the following consequence

\begin{theorem}[The Chow-Rashevski\"i theorem]
Assume that $E$ is bracket-generating. Then for every pair of points in $x,y \in M$, there exist a horizontal curve connecting these points.
\end{theorem}

\begin{example}
We consider $M = \mathbb{R}^3$ with a subbundle $E  = \spn \{ X, Y \}$.
\begin{enumerate}[\rm (a)]
\item If $X = \partial_x$ and $Y = \partial_y$, then $E$ is not bracket-generating.
\item If $X = \partial_x - \frac{1}{2} x \partial_z$ and $Y = \partial_y + \frac{1}{2} y \partial_z$, then $E$ is bracket-generating since
$$[X,Y] = \partial_z.$$
\item If $X = \partial_x$ and $Y = \partial_y + \frac{1}{2} x^2 \partial_z$, then $E$ is bracket-generating since
$$[X,Y] = x \partial_z, \qquad [X,[X,Y]]= \partial_z.$$
\end{enumerate}
\end{example}

We define $\underline{E}^1 = \Gamma(E)$, $\underline{E}^{k+1} = \underline{E}^k + [\underline{E}, \underline{E}^k]$ and write $E_x^k = \underline{E}^k|_x$.
We call the minimal number $s(x)$ such that $E^{s(x)}_x = T_xM$ \emph{the step} of $E$ at $x$. For each point of $x$, we define
$$\mathfrak{G}(x) = ( \rank E_x^1, \rank E_x^2, \dots, \rank E^{s(x)}_x),$$
as \emph{growth vector} at $x \in M$. We say that $E$ is \emph{equiregular} if $\mathfrak{G}$ is independent of $x \in M$.

\begin{example}
We again consider $M = \mathbb{R}^3$ with a subbundle $E  = \spn \{ X, Y \}$.
\begin{enumerate}[\rm (a)]
\item If $X = \partial_x - \frac{1}{2} x \partial_z$ and $Y = \partial_y + \frac{1}{2} y \partial_z$, with
$[X,Y] = \partial_z$ then $\mathfrak{G} = (2,3)$ so $E$ is equiregular.
\item If $X = \partial_x$ and $Y = \partial_y + \frac{1}{2} x^2 \partial_z$, we have brackets
$$[X,Y] = x \partial_z, \qquad [X,[X,Y]]= \partial_z, \qquad \text{and growth vector} \qquad \mathfrak{G}(x,y,z) = \left\{ \begin{array}{ll}
(2,3) & x \neq 0, \\ (2,2,3) & x =0
\end{array}\right.$$
so $E$ is not equiregular.
\end{enumerate}
\end{example}

\subsection{Sub-Riemannian structures}
On a connected manifold $M$, \emph{a sub-Riemannian structure} is a pair $(E,g)$ where $E$ is a subbundle of the tangent bundle and $g = \langle \cdot , \cdot \rangle_g$ is a smoothly varying inner product
on $E$. The triple $(M, E,g)$ is called \emph{a sub-Riemannian manifold}. For a horizontal curve $\gamma:[a,b] \to M$, we define the length
$$L(\gamma) = \int_a^b \langle \dot \gamma(t), \dot \gamma(t) \rangle^{1/2}_g \, dt.$$
We define the Carnot-Carath\'eodory distance or the sub-Riemannian distance on $M$ by
$$d_g(x,y) = \inf \{ L(\gamma) \, : \, \text{$\gamma$ is a horizontal curve from $x$ and $y$} \}.$$

\begin{example}
Let $M = \mathbb{R}^3$ and define
$$X = \partial_x - \frac{1}{2} y \partial_z, \qquad Y = \partial_y + \frac{1}{2} x \partial_z.$$
Write $E = \spn\{ X, Y\}$ and $\langle X, X \rangle_g = \langle Y, Y\rangle_g = 1$, $\langle X,Y \rangle_g = 0$.
The sub-Riemannian manifold $(M, E, g)$ is called the Heisenberg group and is an example of a Carnot group,
which we will return to in Section~\ref{sec:Carnot}.
\begin{center}
\begin{tabular}{cc}
\includegraphics[width=6cm]{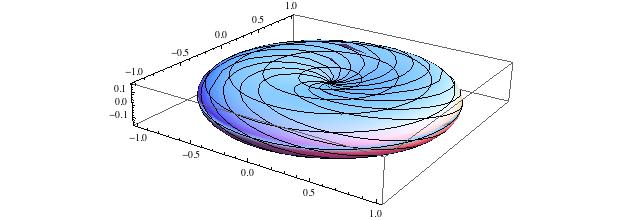} & \includegraphics[width=3cm]{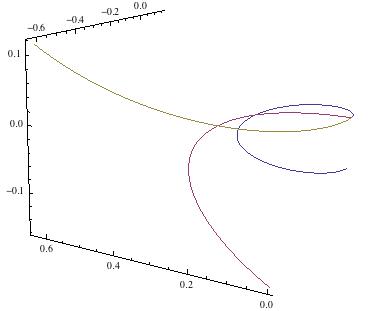} \\
Unit ball in the Heisenberg groups & Examples of geodesics in the Heisenberg group
\end{tabular}
\end{center}
\end{example}

\begin{example}[The Hopf fibration] \label{ex:Hopf}
We can consider the Hopf fibration as the surjective map from $S^3 \subseteq \mathbb{C}^2$ to the corresponding complex projective line $\mathbb{C}P^1 \cong S^2$, given by
$$\pi: (z,w) \mapsto [z,w] \text{ (equivalence class)}.$$
Let $S^3$ have its usual metric and define $E = (\ker \pi_*)^\perp$ with $g$ the restriction of this metric to $E$. We then have a sub-Riemannian manifold $(S^3, E, g)$.

Identify $S^3$ with the Lie group $\SU(2)$ whose Lie algebra $\su(2)$ is spanned by three elements $X,Y, Z$, having cyclic bracket relations
$$[X,Y] =Z, \qquad [Y, Z] = X, \qquad [Z,X] = Y.$$
Define $K = \exp(\mathbb{R} Z)$. We can see $\pi$ as the map $\SU(2) \mapsto \SU(2)/K$, $E = \spn \{ X, Y \}$ and $\langle X, X \rangle_g = \langle Y, Y \rangle_g = 1$ and $\langle X, Y \rangle_g = 0$. 
\end{example}

An isometry of sub-Riemannian manifolds $\Phi : (M, E, g) \to (\tilde M, \tilde E, \tilde g)$ is a homeomorphism such that $d_{\tilde g}(\Phi(x), \Phi(y)) = d_{g}(x,y)$. From \cite{CaLD16}, we know that when $E$ and $\tilde E$ are equiregular, then $\Phi$ is diffeomorphism $\Phi: M \to \tilde M$ with $\Phi_* E \subseteq \tilde E$ and such that
$$\langle \Phi_* v, \Phi_* w \rangle_{\tilde g} =\langle v, w \rangle_g .$$
When $(M,E,g) = (\tilde M, \tilde E, \tilde g)$, these isometries together form a Lie group.

\subsection{The flat spaces: Carnot groups} \label{sec:Carnot}
Let $\mathfrak{g}$ be a nilpotent Lie algebra. A stratification of a nilpotent Lie algebra is a decomposition
$$\mathfrak{g} = \mathfrak{g}_1 \oplus \mathfrak{g}_2 \oplus \cdots \oplus \mathfrak{g}_s,$$
such that
$$[\mathfrak{g}_1, \mathfrak{g}_k] = \mathfrak{g}_{k+1}, \qquad [\mathfrak{g}_1, \mathfrak{g}_s] = 0.$$
A Carnot algebra is a Lie algebra $\mathfrak{g}$ with a stratification and an inner product $\langle \cdot , \cdot \rangle$ on $\mathfrak{g}_1$.

Let $G$ be the corresponding simply connected Lie group. We can then define a subbundle $E$ by left translation of $\mathfrak{g}_1$. In other words, $E_a = a \cdot \mathfrak{g}_1$.
We can define a sub-Riemannian metric $g$ on $E$ by
$$\langle v, w \rangle_g = \langle a^{-1} \cdot v, a^{-1} \cdot w \rangle.$$
The sub-Riemannian manifold $(G, E, g)$ is then called \emph{a Carnot group}.

\begin{example}[The $n$-th Heisenberg group] \label{ex:nHeis}
Consider the Lie algebra $\mathfrak{g} = \mathfrak{g}_1 \oplus \mathfrak{g}_2$, where
$$\mathfrak{g}_1 = \spn\{ X_1, \dots, X_n, Y_1, \dots, Y_n \},\qquad \mathfrak{g}_2 = \spn\{ Z\}.$$
$$[X_i, X_j] = 0, \qquad [Y_i, Y_j] = 0, \qquad [X_i, Y_j] = \delta_{ij}Z.$$
The standard metric on $\mathfrak{g}_1$ is such that $X_1, \dots, X_n, Y_1, \dots, Y_n$ form an orthonormal basis. We can realize the corresponding group $G = \{ (x, y, z) \in \mathbb{R}^n \times \mathbb{R}^n \times \mathbb{R}$, with
$$(x, y, z) \cdot (\tilde x, \tilde y, \tilde z) = \left(x+ \tilde x, y + \tilde y, z + \tilde z+ \frac{1}{2}(\langle x, \tilde y\rangle - \langle y, \tilde x \rangle )\right).$$
We have that the corresponding left invariant vector fields are given by
$$X_j = \partial_{x^j} - \frac{1}{2} y^j \partial_z, \qquad Y_j = \partial_{y^j} + \frac{1}{2} x^j \partial_z, \qquad Z = \partial_z.$$
\end{example}
Since the Lie group is nilpotent, we have that the exponential map $\exp: \mathfrak{g} \to G$ is a diffeomorphism. The group operation is then given by the Baker-Campbell-Hausdorff formula {\small
$$\exp(A) \cdot \exp(B) = \exp\left(A + B + \frac{1}{2} [A,B] + \cdots \right) = \exp\left(\sum_{j=1}^s \frac{(-1)^{j-1}}{j} \sum_{\begin{subarray} q_i + p_i >0 \\ i=1, \dots, j
\end{subarray}} \frac{[X^{p_1} Y^{q_1} \cdots X^{p_j}Y^{q_j}]}{\sum_{i=1}^j (p_i+q_i) \prod_{i=1}^j p_i ! q_i!} \right).$$}
This sum is finite and gives us a group operation. For Carnot groups, we often use exponential coordinates and so the identity is often denoted $0$ rather than $1$.

\begin{example}
The Engel algebra is the algebra given by non-tivial brackets
$$\mathfrak{g} = \mathfrak{g}_1 \oplus \mathfrak{g}_2 \oplus \mathfrak{g}_3 = \spn\{ X, Y\} \oplus \spn\{ Z\} \oplus \spn \{ W \}.$$
$$[X, Y] = Z, \qquad [X, Z] = W.$$
The corresponding Engel group can be considered as $\mathbb{R}^4$ with the product
\begin{align*}
& (x,y,z,w) \cdot (\tilde x, \tilde y, \tilde z, \tilde w) \\
& = \left(x+ \tilde x,y+ \tilde y, z + \tilde z +\frac{1}{2} (x \tilde y - y \tilde x), w + \tilde w + \frac{1}{2} (x \tilde z - z \tilde x) + \frac{1}{12}(x^2 \tilde y + \tilde x^2 y - (y+\tilde y) x\tilde x) \right) .
\end{align*}
\end{example}

\begin{example} \label{ex:Free}
Let $V$ be an inner product space and consider $T^s(V)$ be the truncated tensor algebra of $V$, where we have divided out any tensor product of length $s+1$ or longer. We can then obtain a Lie algebra $\mathfrak{f} =\free_s(V)$ by dividing out the elements $x \otimes y + y \otimes x$ and $x \otimes y \otimes z + y \otimes z \otimes x + z \otimes x \otimes y$. This has a natural stratification $\mathfrak{f} = \mathfrak{f}_1 \oplus \cdots \oplus \mathfrak{f}_s$, where $\mathfrak{f}_j$ is the image under the quotient of $V^{\otimes j}$. Since $\mathfrak{f}_1 =V$ is an inner product space, we have that $\mathfrak{f}$ has the structure of a Carnot algebra.

Observe that if $\mathfrak{g} =\mathfrak{g}_1 \oplus \cdots \oplus \mathfrak{g}_s$ is a general Carnot algebra, then it will be a quotient of $\free_s(\mathfrak{g}_1)$ by dividing out the additional relations.
\end{example}

A special property for Carnot groups is that they have dilations. If $\mathfrak{g}$ is a Carnot algebra, for any $r >0$, we define a linear map $\dil_r: \mathfrak{g} \to \mathfrak{g}$ by
$$\dil_r(A) = r^k A \qquad \text{for any $A \in \mathfrak{g}_k$}.$$

\begin{lemma}
\begin{enumerate}[\rm (a)]
\item For a stratified Lie algebra,
$$[\mathfrak{g}_i, \mathfrak{g}_j] \subseteq \mathfrak{g}_{i+j},$$
where we interpret $\mathfrak{g}_{i+j} =0$ if $i+j > s$.
\item The map $\dil_r$ is a Lie algebra isomorphism.
\end{enumerate}
\end{lemma}
\begin{proof}
\begin{enumerate}[\rm (a)]
\item We will do induction in $i$. For $i =1$, we know that
$[\mathfrak{g}_1, \mathfrak{g}_j] \subseteq \mathfrak{g}_{j+1},$
for any $j \geq 1$. Next, assume that for some $i$, we have $[\mathfrak{g}_i, \mathfrak{g}_j] \subseteq \mathfrak{g}_{i+j}$ for any $j \geq 1$. Then any element in $\mathfrak{g}_{i+1}$ can be written as $[A,B]$ with $A \in \mathfrak{g}_1$, $B \in \mathfrak{g}_i$, and with $C \in \mathfrak{g}_{j}$, we have that
$$[[A,B], C] = [[A,C], B] + [A, [B,C]].$$
We note that $[B,C] \in \mathfrak{g}_{i+j}$ by induction hypothesis and so $[A,[B,C]] \in \mathfrak{g}_{i+j+1}$. Similarly, $[A,C] \in \mathfrak{g}_{j+1}$, so by the induction hypothesis $[[A,C],B] \in \mathfrak{g}_{i+j+1}$. The result now follows.
\item If $A \in \mathfrak{g}_i$, $B \in \mathfrak{g}_j$, then $[A,B] \in \mathfrak{g}_{i+j}$ and so
$$[\dil_r A, \dil_r B] = [r^i A, r^j B] = r^{i+j} [A,B] = \dil_r[A,B],$$
so $\dil_r$ is a Lie algebra homomorphism that is also invertible with $\dil_r^{-1} = \dil_{1/r}$. \qedhere
\end{enumerate}
\end{proof}

Since $G$ is simply connected, there is a Lie group isomorphism $\Dil_r: G \to G$ such that
$$\Dil_r(\exp(A)) = \exp(\dil_r A).$$

\begin{proposition}
\begin{enumerate}[\rm (a)]
\item If $\gamma$ is a horizontal curve in $G$, then $L(\Dil_r\gamma) = rL(\gamma)$.
\item For any $x,y \in G$, $d_g(\Dil_r(x), \Dil_r(y)) = rd_g(x,y)$.
\end{enumerate}
\end{proposition}
\begin{proof}
\begin{enumerate}[\rm (a)]
\item Since the structure is left invariant, we have that
$$\langle \dot \gamma, \dot \gamma\rangle_g = \langle \gamma^{-1} \cdot \dot \gamma, \gamma^{-1} \cdot \dot \gamma \rangle.$$
Furthermore,
\begin{align*}
& (\Dil_r \gamma(t) )^{-1} \cdot \frac{d}{dt} (\Dil_r \gamma) = \frac{d}{ds} (\Dil_r \gamma(t) )^{-1} \cdot (\Dil_r \gamma(t+s) )|_{s=0} \\
& = \frac{d}{ds} \Dil_r (\gamma(t)^{-1} \cdot \gamma(t+s) )|_{s=0} = \dil_r (\gamma^{-1}(t) \cdot \dot \gamma(t)).
\end{align*}
Finally, since $\gamma(t)^{-1} \cdot \dot \gamma(t)$ is in $\mathfrak{g}_1$ for almost ever $t$, we have for the same values of $t$,
$$\dil_r (\gamma^{-1}(t) \cdot \dot \gamma(t)) = r \gamma^{-1}(t) \cdot \dot \gamma(t).$$
Computing the length gives us
\begin{align*}
L(\Dil_r \gamma) & = \int_a^b \left\langle \frac{d}{dt} \Dil_r \gamma(t), \frac{d}{dt} \Dil_r \gamma(t) \right\rangle^{1/2}_g \, dt \\
& = \int_a^b r \left\langle \gamma(t)^{-1} \cdot \dot \gamma(t), \gamma(t)^{-1} \cdot \dot \gamma(t) \right\rangle^{1/2}_g \, dt = r L(\gamma).
\end{align*}
\item Note that $\Dil_r$ has inverse $\Dil_{1/r}$, so $\gamma \mapsto \Dil_r \gamma$ is an invertible map from all horizontal curves from $x$ to $y$ to all horizontal curves from $\Dil_r(x)$ and $\Dil_r(y)$, and furthermore, $\Dil_r \gamma$ has length $rL(\gamma)$. It follows that
$$d_g(\Dil_r(x), \Dil_r(y)) = \inf \{ rL(\gamma) \, : \, \text{$\gamma$ is a horizontal curve from $x$ and $y$} \} = r d_g(x,y).\qedhere$$
\end{enumerate}
\end{proof}

Since dilations satisfy the property $d_g(\Dil_r(x), \Dil_r(y)) = rd_g(x,y)$, they are the analogue of scaling by a constant in the Euclidean spaces. We can interpret this fact as saying that the Euclidean space and the Carnot group look the same when we ``zoom in''. Note that zooming in on any Riemannian manifold by replacing a distance $d_g$ by $rd_g$ and letting $r$ become large, it will converge towards the Euclidean space, not only topologically, but also metrically. Sub-Riemannian manifolds $(M, E, g)$ with $E$ equiregular will all look like Carnot groups when we ``zoom in'', but not necessarily the same Carnot group at every point. This zooming in can be made formal using Gromov-Hausdorff convergence of pointed metric spaces. For more details, we refer to~\cite{Bel96}.

\begin{remark} In the definition of Carnot groups, some authors include the possibility of just having a norm on the space $\mathfrak{g}_1$, leading to sub-Finsler geometry. This part of the theory has yet to be fully developed.
\end{remark}

\subsection{Isometries of Carnot groups}
Let $l_a:G \to G$ be the left translation $l_a(b) = a \cdot b$, $a,b \in G$. Since the sub-Riemannian structure is left invariant, each such map will be an isometry. Next, we consider a Lie group isomorphism $\Phi:G \to G$. If we consider the induced Lie algebra isomorphism $\varphi = \Phi_{*,1}: \mathfrak{g} \to \mathfrak{g}$, then for $\Phi$ to be an isometry, we must have that $\varphi$ maps $\mathfrak{g}_1$ to itself isometrically. We will then call $\varphi$ a \emph{Carnot algebra isometry}.

\begin{theorem} \cite{DO16} \label{th:IsoCarnot}
Assume that $\Phi: G \to G$ is an isometry of a Carnot group with $\Phi(1) = a$. Then $l_{a^{-1}} \circ \Phi$ is both an isometry and a Lie group isomorphism. In other words, any isometry of a Carnot group is the composition of a left translation and a Lie group isomorphism. 
\end{theorem}
From this result, it follows that after left translation, every isometry $\Phi$ is a Lie group isomorphism uniquely determined by the isometry $\varphi = \Phi_{*,1}$ of the Carnot algebra $\mathfrak{g}$ which in turn is determined by the linear isometry $q = \varphi |_{\mathfrak{g}_1}$. However, if $q: \mathfrak{g}_1 \to \mathfrak{g}_1$ is a linear isometry, it cannot necessarily be extended to a Carnot algebra isometry.
\begin{example}
We consider the Carnot algebra isometries of the examples of Section~\ref{sec:Carnot}.
\begin{enumerate}[\rm (a)]
\item For the free nilpotent Lie algebra $\free_s(V)$, any linear isometry of $V$ can be extended to an isometry of the Carnot algebra.
\item For the Heisenberg algebra $\mathfrak{g} = \mathfrak{g}_1 \oplus \mathfrak{g}_2 = \spn\{ X_1, \dots, X_n, Y_1, \dots, Y_n\} \oplus \spn \{Z\}$ as described in Example~\ref{ex:nHeis}. Then an isometry $q: \mathfrak{g}_1 \to \mathfrak{g}_1$ can be extended to an Carnot algebra isometry if it preserves the $2$-vector
$$\sum_{i=1}^n X_i \wedge Y_i.$$
Hence, it can be considered as the intersection of $\Ort(2n)$ and $\Sp(n)$, which is isomorphic to $U(n)$.
\item The Engel algebra has an isometry group isomorphic to $\mathbb{Z}/2\mathbb{Z}$, as only $\pm \id_{\mathfrak{g}_1}$ can be extended to a Lie algebra isometry.
\end{enumerate}
\end{example}

\subsection{Symbols}
Let $(M, E, g )$ be a sub-Riemannian manifold. Assume that $E$ is equiregular. Then we have a growing flag
$$E^0 = 0 \subseteq E^1 = E \subseteq E^2 \subseteq \cdots \subseteq E^s = TM,$$
of subbundles. Define
$$\symb_x   = E_x \oplus E^2_x/E_x \oplus \cdots \oplus E^s_x/ E^{s-1}_x.$$
and a Lie algebra structure on $\symb_x$ by
$$\Lbra X_x \mod E^{i-1}_x, Y_x \mod E^{j-1}_x\Rbra = [X,Y]|_x \mod E^{i+j-1}_x, \qquad X_x \in E^{i}_x, Y_x \in E^j_x.$$
where $X$ and $Y$ are any vector field extending $X_x$ and $Y_x$. Since $[X,Y]|_x \mod E^{i+j-1}_x$ is tensorial in $X$ and~$Y$, this map does not depend on the extensions of $X_x$ and $Y_x$. This makes $(\symb_x, \Lbra \cdot, \cdot \Rbra)$ into a nilpotent Lie algebra with a stratification $\symb_{x,j} = E^j_x/E^{j-1}_x$ and with an inner product on $\symb_{x,1} = E_x$; in other words, a Carnot algebra.

From the Carnot algebra $\symb_x$, we get a corresponding Carnot groups $(\Symb_x, \tilde E, \tilde g)$. This Carnot group is what $(M,E,g)$ looks like when we ``zoom in'' close to $x$. We refer to \cite{Gro96,Bel96} for details.

We say that $(M,E, g)$ has \emph{constant symbol} $\mathfrak{g}$ if $\symb_x$ is isometric to Carnot algebra $\mathfrak{g}$ for any $x \in M$.

\begin{example}[Hopf fibration]
We consider again $\SU(2)$ with the sub-Riemannian structure as in Example~\ref{ex:Hopf}.
Then for any point $\symb_x = \spn \{ X, Y \} \oplus \spn \{ Z \mod E\}$, with
$$\Lbra X,Y\Rbra = Z \mod E, \qquad \Lbra X,Z \mod E\Rbra = \Lbra Y,Z \mod E\Rbra =0.$$
We see that $\symb_x$ is the Heisenberg algebra.
\end{example}

\begin{example}
Consider $\mathbb{R}^5$ with coordinates $(x_1, x_2, y_1, y_2, z)$ with $(E,g)$ given by an orthonormal basis
$$A_1 = \partial_{x_1}, \qquad B_1 = (1+y^2_1)(\partial_{y_1} + x_1 \partial_{x_1}).$$
$$A_2 = \partial_{x_2}, \qquad B_2 = \partial_{y_2} + x_1 \partial_{x_1}.$$
Then $\symb_{x_1, x_2, y_1, y_2, z}$ is isometric to the 2nd Heisenberg algebra
$$[X_1, Y_1] = [X_2, Y_2] = Z,$$
but with an orthonormal basis given by $\sqrt{1+y^2_1} X_1$,$\sqrt{1+y^2_1} Y_2$, $X_2$, $Y_2$, where $y_1$ is now considered as a constant. These are not isometric for different values of $y_1^2$.  
\end{example}

If there is just one Carnot algebra in the class of growth vectors, then all sub-Riemannian manifolds with that growth vector will have constant symbol.

\begin{example}
Consider two 2-dimensional Riemannian manifolds $\Sigma$ and $\tilde \Sigma$, whose Gaussian curvature never coinsides. On $\hat M = \Ort(\Sigma) \times \Ort(\tilde \Sigma)$, consider the sub-Riemannian structure $(\hat E, \hat g)$ with orthogonal vector fields $H_{e_1} + \tilde H_{e_1}$ and $H_{e_2} + \tilde H_{e_2}$. Here $H_{e_j}$ and $\tilde H_{e_j}$ corresponds to the canonical horizontal vector fields of the Levi-Civita connection on the orthonormal frame bundle of respectively $\Sigma$ and $\tilde \Sigma$, as defined in Section~\ref{sec:ConFrame}. Define $M = \hat M/\Ort(n)$ as the quotient the diagonal action and let $(E,g)$ be the induced sub-Riemannian structure on $(M,E,g)$. Then $(M,E,g)$ has growth vector $(2,3,5)$. Since there is only one Carnot group with this growth vector, these have all constant symbol.

This sub-Riemannian manifold correspond to the dynamical system of rolling $\Sigma$ on $\tilde \Sigma$ without twisting or slipping. For details, see \cite{GGMSL12,Gro12}.
\end{example}

\subsection{Affine connections: Why we cannot have everything we want}
Let $(M, E, g)$ be a sub-Riemannian manifold. Let $\nabla$ be an affine connection of $TM$. If we want $\nabla$ to preserve the sub-Riemannian structure, we should demand that parallel transport take orthonormal frames of $E$ to orthonormal frames of $E$. We therefore introduce the following definition.
\begin{definition}
A connection $\nabla$ is said to be compatible with the sub-Riemannian structure $(E,g)$ if
\begin{enumerate}[\rm (i)]
\item for any $Y \in \Gamma(H)$ and $X \in \Gamma(TM)$, we have that $\nabla_X Y$ takes values in $E$ as well,
\item for any $Y_1, Y_2 \in \Gamma(H)$ and $X \in \Gamma(TM)$, we have
$$X \langle Y_1, Y_2 \rangle_g = \langle \nabla_X Y_1, Y_2 \rangle_g + \langle Y_1, \nabla_X Y_2 \rangle_g.$$
\end{enumerate}
\end{definition}
We can also describe this property using parallel transport.
\begin{lemma}
A connection $\nabla$ is compatible with the sub-Riemannian metric if and only if $\nabla$ always takes an orthonormal basis of $E$ to another orthonormal basis of $E$.
\end{lemma}
\begin{proof}
Let $E$ have rank $k$.
First assume that (i) and (ii) holds for $\nabla$. Let $\gamma(t)$ be a curve in $M$, and let $Y_1(t), \dots, Y_k(t)$ be an orthonormal frame of vector fields along $\gamma(t)$. We  see that
$$\nabla_{\dot \gamma} \sum_{j=1}^k a_j(t) Y_j(t) = \sum_{i=1}^k (\dot a_j(t) Y_j(t) + a_j(t) \nabla_{\dot \gamma} Y_j(t)).$$
By property (i), there are coefficients $c_{i}^j$ such that $\nabla_{\dot \gamma} Y_i = \sum_{k=1}^k c_{j}^i Y_i$, and we can write
$$\nabla_{\dot \gamma} \sum_{j=1}^n a_j(t) Y_j(t) = \sum_{i=1}^k \left(\dot a_j(t) + \sum_{i=1}^k a_i(t) c_j^i(t) \right) Y_j(t).$$
We can then solve the equation
$$\frac{d}{dt} a_j = - (c_{j}^i) (a_i),$$
to get a parallel vector field with an arbitrary initial value in $E$. It follows that any vectors in $E$ remains in $E$ under parallel transport. Next, if $Y_1(t), \dots, Y_k(t)$ is a parallel basis along $\gamma$ that has orthonormal initial value. then from~(ii)
$$\frac{d}{dt} \langle Y_i(t), Y_j(t) \rangle_{g} = 0,$$
so they remain orthonormal.

Conversely, we can prove (i) and (ii) from the property of parallel transporting orthonormal bases to orthonormal basis, by using that $\nabla_X Y = \nabla_{\dot \gamma} Z(t)$ for $Z(t) = Y|_{\gamma(t)}$.
\end{proof}

\begin{proposition}
Assume that $E$ is a proper subbundle\footnote{meaning that $E$ is not the full tangent bundle} of $TM$  and that $E$ is bracket-generating. Then there are no torsion-free affine connections that are also compatible with $(E,g)$.
\end{proposition}
This result highlights that we cannot use the same normalization condition for a connection of a sub-Riemannian manifold as we did for a Riemannian manifold.
\begin{proof}
Assume that $\nabla$ is an affine connection that is both compatible with $(E,g)$ and torsion-free. Let $Y_1, Y_2$ be any pair of vector fields with values in $E$. Then by (i), we have that
$$[Y_1, Y_2] = \nabla_{Y_1} Y_2 - \nabla_{Y_2} Y_1,$$
takes values in $E$ as well. This shows that any bracket of vector fields in $E$, also takes values in $E$. Since $E$ is bracket-generating, it follows that $E = TM$. But we assumed that this is not the case, so we have a contradiction.
\end{proof}

\subsection{Sub-Riemannian frame bundles} \label{sec:SRFrame}
From now on, we will restrict ourselves to sub-Riemannian manifolds with constant symbol $\mathfrak{g}$. Let $G$ be the corresponding Carnot group with sub-Riemannian structure $(\tilde E, \tilde g)$. Define $G_0 = \Isom(\mathfrak{g})$ as the Lie group of Cartan isometries with Lie algebra $\mathfrak{g}_0 =\isom(\mathfrak{g})$. We then see that $\mathfrak{g}_0$ consists of derivations of $\mathfrak{g}$ preserving the stratification and whose restriction to $\mathfrak{g}_1$ is skew symmetric. Define a new algebra $\hat{\mathfrak{g}} = \mathfrak{g}_0 \oplus \mathfrak{g}$ such that both $\mathfrak{g}_0$ and $\mathfrak{g}$ are subalgebras and with
$$[D, A] = DA.$$

\begin{proposition}
$\hat{\mathfrak{g}}$ is the Lie algebra of isometries of the isometry group $\hat G= \Isom(G, \tilde E, \tilde g)$.
\end{proposition}
\begin{proof}
Let $\varphi: \mathfrak{g} \to \mathfrak{g}$ be a Carnot algebra isometry with corresponding isometry $\Phi_{\varphi}:G \to G$ that is also a Lie algebra automorphism. By Theorem~\ref{th:IsoCarnot}, we can define an invertible map from $G_0 \times G$ to $\hat G$ by $(\varphi, a) \mapsto \Psi_{\varphi,a} = l_a \circ \Phi_\varphi$. We observe that 
$$\Phi_{\varphi_1} \circ \Phi_{\varphi_2}= \Phi_{\varphi_1  \varphi_2}, \qquad \Phi_\varphi \circ l_{a} = l_{\Phi_\varphi(a)} \circ \Phi_{\varphi}.$$
Hence
\begin{align*}
& \Phi_{\varphi_1, a_1} \circ \Phi_{\varphi_2,a_2} = l_{a_1 \Phi_{\varphi_1}(a_2)} \circ l_{\Phi_{\varphi_1}(a_2^{-1})} \Phi_{\varphi_1} l_{a_2} \Phi_{\varphi_2} = l_{a_1 \Phi_{\varphi_1}(a_2)} \circ \Phi_{\varphi_1 \varphi_2} .
\end{align*}
As such, we can consider $\hat G$ as $G_0 \times G$ with the group operation $(\varphi_1, a_1) \cdot (\varphi_2, a_2) = (\varphi_1 \varphi_2, a_1 \varphi_2(a_2))$. Computing the Lie algebra from this multiplication rule, we get the Lie algebra $\hat{\mathfrak{g}}$.
\end{proof}

Let $(M,E,g)$ be a sub-Riemannian manifold with constant symbol $\mathfrak{g}$. Define a vector bundle of symbols $\symb = (\symb_x)_{x\in M} \to M$ over $M$. We will call this \emph{the non-holonomic tangent bundle}. We now define non-holonomic frame as a Carnot isomorphism $u: \mathfrak{g} \to \symb_x$. We will write a the space of all such frames as~$\mathscr{F}_x$. We again have a right action of $G_0$ on $\mathscr{F}_x$ by precomposition. This gives us a principal bundle
$$G_0 \to \mathscr{F} \to M.$$
Let $\psi$ be a Cartan connection with values in $\hat{\mathfrak{g}}$. We can write it as $\psi = (\omega, \theta)$ where $\omega$ and $\theta$ have values in respectively $\mathfrak{g}_0$ and $\mathfrak{g}$. Observe the following.
\begin{enumerate}[$\bullet$]
\item The form $\omega: T\scrF \to \mathfrak{g}_0$ is a principal connection on the bundle $\mathscr{F}$. It corresponds to an affine connection $\tilde \nabla$ on $\symb$ such that parallel transport maps are Carnot algebra isometries.
\item Write $\calH = \ker \omega$ and define $h_u v$ as the horizontal lift of $v \in T_xM$ to $T_u \scrF$, $u \in \scrF$ with respect to this connection. The form $\theta: T\scrF \to \mathfrak{g}$ corresponds to a vector bundle isomorphism $I: TM \to \symb$ in the following way. For any $v \in T_x M$, and $u\in \mathscr{F}_x$, 
\begin{equation}
\label{Imap} I: v \mapsto h_u v \in T_u \scrF \mapsto \theta(h_u v) \in \mathfrak{g} \mapsto u (\theta(h_u v)) \in \symb_x.\end{equation}
We can see this map as a way choosing complements to
$$E^{k+1} = V^{k+1} \oplus E^{k},$$
by $I^{-1}(E^{k+1}/E^k)$. We will call such an $I$ \emph{an $E$-grading}.
\item From the previous structures, we can define a connection $\nabla = I^{-1} \tilde\nabla I$ on $TM$. 
\end{enumerate}
The above definition makes sense because of the following.
\begin{lemma}
The map $I$ is well defined, that is, the mapping $I_x = I |_{T_xM}$ in \eqref{Imap} does not depend on the choice of element $u \in \mathscr{F}_x$.
\end{lemma}

\begin{proof}
If $u \in \mathscr{F}_x$, then any other element is on the form $u \cdot a = u \circ a$ with $a \in G_0$, we then observe that
$$(u \cdot a) (\theta(h_{u \cdot a} v)) = u(a( \theta(h_{u}v \cdot a))) =  u(a( a^{-1} \theta(h_{u}v \cdot a))) = u(\theta(h_{u}v)) . \qedhere$$
\end{proof}

In summary, any choice of Cartan connection on $\mathscr{F}$ gives us an identification $I: TM \to \symb$ and connection $\nabla$ on $TM$. The connection $\nabla$ has the following properties.
\begin{enumerate}[\rm (i)]
\item The connection $\nabla$ is compatible with sub-Riemannian metric $(E,g)$,
\item Recall that $I$ gives a decomposition $TM = V_1 \oplus \cdots \oplus V_s$ such that $E^{k+1} = E^k \oplus V^{k+1}$. Each vector bundle $V^k$ will be parallel with respect to $\nabla$.
\item Define a tensor $\mathbb{T}$ such that
$$\mathbb{T}(v,w) = - I^{-1} \Lbra Iv, Iw \Rbra.$$
Then $\nabla \mathbb{T} =0$. We can consider $\mathbb{T}$ as the degree-zero part of the torsion $T$ of $\nabla$ since
$$\mathbb{T}(v,w) = \sum_{i,j=1} \pr_{V_{i+j}} T(\pr_{V_i} v, \pr_{V_j} w).$$
\end{enumerate}
We will call $(E, g, I)$ a graded sub-Riemannian manifold and say that a connection satisfying the above is \emph{strongly compatible} with $(E,g,I)$. We refer to \cite{Gro20b} for details, where also the following result is found.
\begin{theorem}[Partial flatness theorem]
If $\nabla$ is strongly compatible with $(E,g,I)$, and if
\begin{equation} \label{flat} R =0, \qquad T = \mathbb{T},\end{equation}
then $(M,E,g)$ is locally isometric to a Carnot group.
\end{theorem}
To make the converse statement, we need a canonical choice for such gradings $I$ and connections $\nabla$.

\subsection{Canoncial grading and connection}
Let $(M,H,g)$ be a sub-Riemannian manifold with constant symbol $\mathfrak{g}$, and with sub-Riemannian frame bundle $G_0 \to \scrF \to M$. Then on $\scrF$, there is a canonical condition for the curvature given by Morimoto \cite{Mor93,Mor08}. Let us introduce some preliminary concepts below in order to understand this theory.
\begin{enumerate}[$\bullet$]
\item On the set of \underline{linear} $k$-forms on $\mathfrak{g}$ with values in $\hat{\mathfrak{g}} = \mathfrak{g}_0 \oplus \mathfrak{g}$, we define the Spencer differential/the Lie algebra cohomology differential by $\partial: \wedge^k \mathfrak{g}^* \otimes \hat{\mathfrak{g}} \to \wedge^{k+1} \mathfrak{g}^* \otimes \hat{\mathfrak{g}}$, $k=1, \dots, n$, by
\begin{align*}
(\partial \alpha)(A_0, \dots, A_k) &= \sum_{i=0}^n (-1)^i [A_i, \alpha(A_0, \dots, \hat A_i, \dots, A_k)] \\
& \qquad + \sum_{i<j} (-1)^{i+j} \alpha([A_i, A_j], A_0, \dots, \hat A_i, \dots, \hat A_j, \dots, A_k).
\end{align*}
\item We can introduce an inner product on $\hat{\mathfrak{g}}$ from the inner product on $\mathfrak{g}_1$ in the following way. Let $\mathbb{B}(A,B) = [A, B]$ be the Lie bracket. First, give $\wedge^2 \mathfrak{g}_1$ an inner product such that such that if $A_1, \dots, A_{n_1}$ is an orthonormal basis of $\mathfrak{g}_1$, then $A_i \wedge A_j$ is an orthonormal basis for $1 \leq i \leq j \leq n$. Then we can then define an inner product on $\mathfrak{g}_2$ by requiring $\mathbb{B}$ to map the orthogonal complement of $\ker \mathbb{B} \cap \wedge^2 \mathfrak{g}_1$ onto $\mathfrak{g}_2$. Next, since $\mathfrak{g}_k$ is the image of $\{ A \wedge B \, : \, A \in \mathfrak{g}_j \wedge B \in \mathfrak{g}_{k-j}, j=1, \dots, k-1\}$ under $\mathbb{B}$, we can proceed iteratively in the same manner to define an inner product on $\mathfrak{g}_k$ for any $k = 3, \dots, s$. In conclusion, there is a unique way to extend the inner product to $\mathfrak{g}$ such that the map $\mathbb{B}|_{(\ker \mathbb{B})^\perp}: (\ker \mathbb{B})^\perp \to \mathfrak{g}_2 \oplus \cdots \oplus \mathfrak{g}_s$ is a linear isometry.

We can see $\mathfrak{g}_1 \oplus \cdots \oplus \mathfrak{g}_s$ as the image of a surjection from the truncated tensor algebra $T^s(\mathfrak{g}_1)$ by dividing out the additional relations. Induce the inner product from this surjection $P: T^s(\mathfrak{g}_1) \to \mathfrak{g}$ by requiring $P$ to map $(\ker P)^\perp$ onto $\mathfrak{g}$ isometrically.

Next, any element $D \in \mathfrak{g}_0$ is a map from $\mathfrak{g}$ to itself and can be considered as an element in $\mathfrak{g}^* \otimes \mathfrak{g}$, the latter having an induced inner product from the previous statement.
\item We finally write $\partial^*$ for the dual of $\partial$ with respect to the inner product on $\mathfrak{g}$.
\end{enumerate}
With these definitions in place, we have the following result \cite{Mor08}.
\begin{theorem}
There is a unique $(\hat{\mathfrak{g}},\mathfrak{g}_0)$-Cartan connection $\psi: T\mathscr{F} \to \hat{\mathfrak{g}}$ with curvature $\kappa: \scrF \to \wedge^2 \mathfrak{g}^* \otimes \hat{\mathfrak{g}}$ satisfying $\partial^* \kappa = 0.$
\end{theorem}
Following \cite{Gro20b}, we can use this result to define a canonical choice of grading $I$ and connection $\nabla$. Let us first introduce some definition,
\begin{enumerate}[$\bullet$]
\item By $I$ we can get a Carnot algebra structure on each $T_x M$. We can extend the sub-Riemannian metric $g$ to a Riemannian metric $g_I$, by using the argument in Section~\ref{sec:SRFrame} (ii).
\item Define a subbundle $\mathfrak{s}$ of $\End TM$ consisting of isometry algebras $\mathfrak{s}_x = \isom(T_xM)$ on each fiber.
\item Finally define $\chi: TM \to \wedge^2 TM$ by $\langle \chi(v), w_1 \wedge w_2 \rangle_{g_I} = - \langle v, \mathbb{T}(w_1, w_2) \rangle_{g_I}$. In other words, $\mathfrak{s}_x = I^{-1} \mathfrak{g}_0 I$.
\end{enumerate}
\begin{theorem} \label{th:MorimotoGC}
There is a unique grading $I$ and connection $\nabla$ that is strongly compatible with $(M, E,g,I)$ and whose torsion $T$ and curvature $R$ satisfies for any $D \in \mathfrak{s}$ and any $v \in V^{i}$, $w \in V^{j}$ with $0 \leq j < i \leq s$,
\begin{align} \label{Rcond}
\langle R(\chi(v)), D \rangle_{g_I} & = \langle T(v, \cdot), D \rangle_{g_I}, \\ \label{Tcond}
\langle T(\chi(v)), w \rangle_{g_I} & =   -  \langle T(v, \cdot) , \mathbb{T}(w, \cdot) \rangle_{g_I} .
\end{align}
\end{theorem}
We can now state the full flatness theorem.
\begin{theorem}
Let $(M,E,g)$ be a sub-Riemannian manifold with constant symbol and let $I$ and $\nabla$ be the canonical grading and connections. Then $(M,E,g)$ is locally isometric to a Carnot group if and only if the curvature and torsion of $\nabla$ satisfies \eqref{flat}.
\end{theorem}
We will look at the canonical grading and curvature in three concrete examples.

\section{Example: Engel-type sub-Riemannian manifolds} \label{sec:Engel}
Let us consider the Engel algebra given by $\mathfrak{g} = \mathfrak{g}_1 \oplus \mathfrak{g}_2 \oplus \mathfrak{g}_3 = \spn \{ A_1, A_2\} \oplus \spn \{ B\} \oplus \spn \{ C\}$ with identities
$$[A_1, A_2] = B, \qquad [A_1, B] = C.$$
We can make it a Carnot algebra by making $A_1$ and $A_2$ into an orthonormal basis. We remark that the isometry algebra for the Engel algebra satisfies $\mathfrak{g}_0 = 0$.

Let $(M, E, g)$ be a sub-Riemannian manifold with growth vector $(2,3,4)$. Then the only possible symbol is the Engel algebra. Since $\mathfrak{s} =0$, we have no condition in \eqref{Rcond} and any strongly compatible connection $\nabla$ will be flat. Consider the symbol at every point
$$\symb = E \oplus E^2/E \oplus E^3/E^2.$$
We note that the map
$$E \otimes E^2/E \to E^3/E^2, \qquad (X_x,Y_x \bmod E) \mapsto \Lbra X_x, Y_x \bmod E \Rbra = [X,Y]_x \bmod E^2,$$
has a one-dimensional kernel $E[0]\subseteq E$. Let $E[1]$ be the orthogonal complement of $E[0]$ in $E$.

Working locally, we can choose a local orthonormal vector field $X_1$ in $E[1]$ and $X_0$ in $E[0]$. Then there is a unique local $\Ann E^2$ section of $\psi$, such that
$$d\psi(X_1, [X_0, X_1]) = -1.$$
Next, define $\theta = -d\psi(X_1, \cdot)$. Then there are unique vector fields $Z$ and $Y$ such that
$$\theta(Z) =0, \qquad \theta(Z)=1, \qquad d\theta(Z, \cdot)|_E = 0,$$
$$\psi(Y) =1, \qquad \psi(Y)=0, \qquad d\theta(Y, \cdot)|_E = 0.$$
From these definitions, it follows that $[X_0,X_1] = Z \mod E$ and that $[X_1,[X_0, X_1]] = Y \mod E^2$. We now have the following result
\begin{theorem} \label{th:Engel}
Assume that
$$[X_0, X_1] = Z + c_0 X_0 + c_1 X_1,$$
Define $X_2 = Z + \frac{2}{5} c_0 X_0 + \frac{1}{2} c_1 X_1$, and if
$$[X_1, X_2] = Y + C_0 X_0 + C_1 X_1 + C_2 Z,$$
we define
\begin{align*}
X_3 & = Y - \frac{1}{5} c_0 Z + \frac{1}{2} \left( C_0 - \frac{1}{5} X_1 c_0 + \frac{3}{25} c_0^2 \right) X_0 \\
& \qquad + \frac{1}{2} \left( C_1 + \frac{1}{5} X_0 c_0 + \psi([X_2 - \tfrac{1}{5} c_0 X_0, Y]) + \frac{1}{10} c_1 c_0 \right)
\end{align*}
Then the canonical grading $I$ is given by $E = V_1$,
$$V^2 = \spn\{ X_2 \}, \qquad V^3 = \spn \{ X_3 \},$$
Finally, the canonical connection $\nabla$ is the unique connection making $X_0$, $X_1$, $X_2$, $X_3$ parallel.
\end{theorem}

\begin{remark}[Complexity of Morimoto's connection]
One observation we can make from Theorem~\ref{th:Engel} is that the connection and grading given above does not seem like the most simple solution. In fact, it would seam like a more natural choice to considerer $V^2$ and $V^3$ spanned by respectively $Z$ and $Y$. We run into similar problems in other examples, and for this reason, gradings and connections presented in Section~\ref{sec:Contact} and Section~\ref{sec:235} will actually be simplified alternatives to those defined by Morimoto's result. The apparent complexity of Cartan connections satisfying Morimoto's normalization conditions has previously also been observed in \cite{AMS19}. These observations could indicate that there exists an alternative normalization condition for Cartan connections on sub-Riemannian manifolds that would be preferable for practical computations.
\end{remark}

\begin{proof}[Proof of Theorem~\ref{th:Engel}]
Let $I$ be the canonical grading. Relative to $I$, define $X_2$ and $X_3$ by,
$$X_2 = \pr_{V_2} [X_0, X_1], \qquad X_3 = \pr_{V^3} [X_1, X_2].$$
Then $X_0$, $X_1$, $X_2$ and $X_3$ form an orthonormal basis of $g_I$.
We see that all vector fields $X_j$ are parallel, so we have that
$$T(X_i, X_j) = - [X_i, X_j].$$
Furthermore, these are orthonormal with respect to $g_I$. We also see that
$$\chi(X_2) = X_0 \wedge X_1, \qquad \chi(X_3) = X_1 \wedge X_2.$$

We can define
$$X_2 = Z - W_1, \qquad \text{ and } \qquad X_3 = Y - \varphi Z - W_2,$$
for vector fields $W_1$, $W_2$ with values in $E$.  Define a form $\alpha_W$ with $W \in \Gamma(E)$ by
$$\alpha_W(Y)  =0, \qquad \alpha_Y(Z) = 0, \qquad \alpha_W(X) = \langle W, X \rangle_g, \quad X \in \Gamma(E).$$
We can then describe $g_I$ such that for $X \in \Gamma(E)$,
$$\langle X, \, \cdot \, \rangle_I = \alpha_X + \alpha_{W_1}(X) \theta + \alpha_{W_2}(X) \psi, \qquad \langle X_2, \, \cdot \, \rangle_I = \theta + \varphi \psi, \qquad \langle X_3, \, \cdot \, \rangle_I = \psi.$$

Using the above notation, we can then look \eqref{Tcond} for all relevant cases:
We first consider degree $2$ and $1$, giving us
\begin{align*}
- \langle T(\chi(X_2)), X_0 \rangle_I & =  \langle [X_0, X_1], X_0 \rangle_I  = \alpha_{X_0}([X_0, X_1]) + \langle X_0, W_1 \rangle \\
& =  \langle T(X_2, X_1), \mathbb{T}(X_0, X_1) \rangle_I   = \langle [X_1, X_2], X_2 \rangle_I  \\
&  = \theta([X_1, X_2]) + \varphi    = - \langle W_1, X_0 \rangle + \varphi, \end{align*}
and
\begin{align*}
-\langle T(\chi(X_2)), X_1 \rangle_I & = \langle [X_0, X_1], X_1 \rangle_I = \alpha_{X_1}([X_0, X_1]) + \langle X_1, W_1 \rangle  \\
& = \langle T(X_2, X_0), \mathbb{T}(X_1, X_0) \rangle_I = \langle [X_0, X_2], X_2 \rangle_I  = \theta( [X_0, X_2] )   = - \langle W_1, X_1 \rangle.
\end{align*}
Putting these results together, we have
\begin{align*}
[X_0,X_1] & = Z - 2 W_1 + \varphi X_0
\end{align*}

If we next consider degree $3$ and $2$, then
\begin{align*}
- \langle T(\chi(X_3)), X_2 \rangle_I & = \langle [X_1, X_2], X_2 \rangle_I =\theta([X_1, X_2]) + \varphi \psi([X_1, X_2])  = \langle X_0, W_1 \rangle + \varphi \\
&  =  \langle T(X_3, X_1), \mathbb{T}(X_2, X_1) \rangle_I  =  \langle [X_1, X_3], X_3 \rangle_I   \\
& = \psi([X_1, X_3]) = - d\psi(X_1, X_3) = \theta(X_3) =- \varphi.
\end{align*}
giving us that
$$-2 \varphi =  \langle W_1, X_0 \rangle.$$
Summarizing we have
$$\varphi = - \frac{1}{2} \langle W_1, X_0 \rangle = \frac{1}{5} c_0, \qquad  \langle W_1, X_1 \rangle = -\frac{1}{2} c_1, \qquad X_2 = Z + \frac{2}{5} c_0 X_0 + \frac{1}{2} c_1 X_1.$$

Finally, using degree $3$ and $1$, we obtain
\begin{align*}
-\langle T(\chi(X_3)), X_0 \rangle_I & = \langle [X_1, X_2], X_0 \rangle = \alpha_{X_0}([X_1, X_2]) + \langle W_1, X_0 \rangle \theta([X_1, X_2]) + \langle W_2, X_0 \rangle  \\
& = \alpha_{X_0}([ X_1, X_2]) + \langle W_1, X_0 \rangle^2 + \langle W_2, X_0 \rangle\\
& = \langle T(X_3, X_1), \mathbb{T}(X_0, X_1) \rangle_I = - \langle [X_1, X_3], X_2 \rangle_I  \\
& = - \theta( [X_1, X_3]) - \varphi \psi([X_1,X_3]) \\
&= X_1 \varphi - \langle W_2, X_0 \rangle  + \varphi^2
\end{align*}
and
\begin{align*}
- \langle T(\chi(X_3)), X_1 \rangle_I & =  \langle [X_1, X_2], X_1 \rangle = \alpha_{X_1}([X_1,X_2]) + \langle W_1, X_1 \rangle \langle W_1, X_0 \rangle + \langle W_2, X_1 \rangle \\
&  = \langle T(X_3, X_0),\mathbb{T}(X_1, X_0) \rangle_I + \langle T(X_3, X_2), \mathbb{T}(X_1, X_2) \rangle_I \\
& = \langle [X_0, X_3],X_2 \rangle_I - \langle [X_2, X_3], X_3 \rangle_I\\
& = \theta([X_0, X_3]) + \varphi \psi([X_0, X_3]) - \psi([X_2, X_3]) \\
& = - X_0 \varphi -\langle W_2, X_1 \rangle + \varphi \psi([X_0, Y])   \\
& \qquad - \psi([X_2, Y ]) - \varphi \langle W_1, X_1 \rangle  - \langle X_1, W_2 \rangle.
\end{align*}
From these we summarize that
\begin{align*}
2\langle W_2, X_0 \rangle & = - \alpha_{X_0}([ X_1, X_2]) - \langle W_1, X_0 \rangle^2  + X_1 \varphi + \varphi^2 \\
& = - \alpha_{X_0}([ X_1, X_2])   + \frac{1}{5} X_1 c_0 - \frac{3}{25} c_0^2
\end{align*}
and 
\begin{align*}
 2\langle W_2, X_1 \rangle 
& = -  X_0 \varphi - \psi([X_2 - \varphi X_0, Y ]) - \varphi \langle W_1, X_1 \rangle \\
& \qquad   - \alpha_{X_1}([X_1,X_2]) - \langle W_1, X_1 \rangle \langle W_1, X_0 \rangle \\
& = - \frac{1}{5} X_0 c_0 - \psi([X_2 - \varphi X_0, Y ]) + \frac{1}{10} c_0 c_1 \\
& \qquad   - \alpha_{X_1}([X_1,X_2]) - \frac{1}{5} c_1 c_0 \\
\end{align*}
Combining these equations, the result follows.
\end{proof}

\section{Example: Contact manifolds} \label{sec:Contact}
\subsection{The Heisenberg algebra} \label{sec:Heisenberg}
As in Example~\ref{ex:nHeis}, the $n$-th Heisenberg algebra is the step~$2$ nilpotent algebra $\mathfrak{h}_n = \mathfrak{g}_{1} \oplus \mathfrak{g}_{2}$ where
$$\mathfrak{g}_{2} = \spn \{Z\}, \qquad \mathfrak{g}_{1} = \spn \{ X_1, \dots, X_n, Y_1, \dots, Y_n \},$$
and with the only non-zero brackets being
$$[X_j, Y_j] = Z, \qquad j =1, 2 \dots, n.$$
For any vector $\lambda = (\lambda_1, \dots, \lambda_n) \in \mathbb{R}^n$ such that
$1=\lambda_1 \leq \lambda_2 \leq \cdots \leq \lambda_n,$
we define
$$\langle X_j , X_j \rangle_{\mathfrak{g}_1} = \langle Y_j , Y_j \rangle_{\mathfrak{g}_1} =\lambda_j^2.$$
We write this Carnot algebra as $\mathfrak{h}_n(\lambda)$. All Carnot algebra structures on $\mathfrak{h}_n$ are isometric to one of these structures. The isometry algebra $\mathfrak{g}_0 = \isom(\mathfrak{h}_n(\lambda))$ is given by
$$\mathfrak{g}_0 = \spn \{ D_{ij} \, : \, i < j , \lambda_i = \lambda_j\} \cup \{ Q_{ij} \, : \, i \leq j , \lambda_i = \lambda_j\}.$$
\begin{align*}
D_{ij}( X_k) & = \delta_{ki} X_j - \delta_{kj} X_i, &  Q_{ij}(X_k) & =  \frac{1}{2} \delta_{ki} T_j + \frac{1}{2} \delta_{kj} Y_i, \\
D_{ij}( Y_k) & = \delta_{ki} Y_j - \delta_{kj} Y_i, &  Q_{ij}(Y_k) & =  -\frac{1}{2} \delta_{ki} X_j - \frac{1}{2} \delta_{kj} X_i,\\
D_{ij} Z &= 0, &  Q_{ij} Z &= 0.\end{align*}

\subsection{Contact manifolds with constant symbols}  
Let $(M, E, g)$ be a sub-Riemannian manifold of dimension $2n+1$ assume that $E$ has rank $2n$. We assume that~$E$ is a contact distribution, that is $X \wedge Y \mapsto \Lbra X, Y \Rbra = [X,Y] \mod E$ is non-degenerate.

Working locally, we can assume that $E = \ker \theta$ for a one-form $\theta$. We normalize $\theta$ by requiring that the maximal imaginary part of the eigenvalues of $d\theta$ is 1. We can then write
$$d\theta(v,w) = \langle v, \Lambda^{-1} J w \rangle_g, \qquad v,w \in E, J^2 = -\id_E.$$
where $\Lambda|_x$ is symmetric on $E$ and has eigenvalues $1 \leq \lambda_{1,x} \leq \cdots \leq \lambda_{x,n}$, each appearing twice. The symbol of $(M,E,g)$ at each point $\mathfrak{h}_n(\lambda_x)$. Hence $(M,E,g)$ only has constant symbol if $\lambda_x = \lambda$ is constant.

Let $1 =\lambda[1] < \lambda[2] < \dots < \lambda[k]$ be the eigenvalues without repetition, with corresponding eigenspace decomposition $E = E[1] \oplus \cdots \oplus E[k]$. Let $\pr[1], \dots, \pr[k]$ be the corresponding projections. We define the Reeb vector field~$Z$ such that
$$\theta(Z) = 1, \qquad d\theta(Z, \cdot) = 0.$$
We define $I$ such that $V^1\oplus V^2 = E \oplus \spn\{Z\}$ with $g_I$ defined such that $Z$ is a unit vector field. We define a tensor $\tau$ and connections $\nabla, \nabla'$ such that for any $Y, Y_1, Y_2\in \Gamma(E)$, $X \in \Gamma(TM)$,
\begin{align*}
\langle \tau_X Y_1, Y_2 \rangle & = \frac{1}{2} \sum_{j=1}^k (\calL_{X- \pr[j]X} g_I)( \pr[j] Y_1, \pr[j] Y_2), \\
\nabla Z & = \nabla' Z =0, \\
\nabla_X Y &= \sum_{j=1}^k \pr[j] \nabla_{\pr[j]X}^{g_I} \pr[j] Y + \sum_{j=1}^k \pr[j] [ X -\pr[j] X,  \pr[j] Y] + \tau_{X} Y, \\
\nabla_X' Y &= \nabla_{X} Y + \frac{1}{2} (\nabla_{X}J)JY, 
\end{align*}
We then have the following result.
\begin{theorem}
For $I$ and $\nabla'$ as above, $\nabla'$ is strongly compatible with $(M,E,g,I)$ for any contact manifold with constant symbol. It is locally isometric to a Carnot group if and only if $\nabla'$ has curvature $R' = 0$ and torsion $T' = d\theta \otimes Z$. 
\end{theorem}
We remark that $\nabla'$ is not the connection of Theorem~\ref{th:MorimotoGC}, but can still be used for flatness theorems of contact manifolds. See \cite{Gro20b} for details.

\section{Example: (2,3,5)-manifolds} \label{sec:235}
We consider finally a sub-Riemannian manifold $(M, E,g)$ with growth vector $\mathfrak{G} = (2,3,5)$. Then there is only one Carnot algebra with this growth vector, namely the free nilpotent Lie algebra $\mathfrak{g} =\free_3(\mathbb{R}^2)$ on a two-dimensional vector space and of step 3 as introduced in Example~\ref{ex:Free}. This can be considered as a Lie algebra
$$\mathfrak{g} =\mathfrak{g}_1 \oplus \mathfrak{g}_2 \oplus \mathfrak{g}_3 = \spn \{ A_1, A_2\} \oplus \spn \{ B\} \oplus \{ C_1, C_2\},$$
with brackets
$$[A_1, A_2] = B, \qquad [A_j, B] = C_j,$$
and with $A_1, A_2$ being an orthonormal basis. Computing the canonical grading and connection in this case becomes very complicated and we refer to \cite{Gro20b}. We can however give the following explicit result for when such a manifold is locally isometric to a Carnot group.

\begin{theorem}[Flatness theorem for $(2,3,5)$-manifolds] \label{th:Flat235} Let $(M,E,g)$ be a sub-Riemannian manifold where $E$ has growth vector $(2,3,5)$.
Let $X_1$, $X_2$ be any local orthonormal basis of $E$. Introduce a basis $X_3 = [X_1, X_2]$, $X_4 = [X_1, X_3]$ and $X_5 = [X_2, X_3]$ with $[X_i, X_j] = \sum_{k=1}^5 c_{ij}^k X_k$. Define vector fields $Z$, $Y_1$ and $Y_2$ by
\begin{align*}
Z & = X_3 + ( c_{23}^3+c_{24}^4 + c_{25}^5) X_1 - ( c_{13}^3+c_{14}^4 + c_{15}^5) X_2, \\
Y_1 & = X_4 - (c_{14}^4 + c_{15}^5) X_3 + (c_{24}^3 - X_2(c_{14}^4 + c_{15}^5) + c_{24}^4 (c_{14}^4 + c_{15}^5) + c_{24}^5(c_{24}^4 + c_{25}^5)) X_1\\
& \qquad -  (c_{14}^3 - X_1(c_{14}^4 + c_{15}^5) + c_{14}^4 (c_{14}^4 + c_{15}^5) + c_{14}^5(c_{24}^4 + c_{25}^5)) X_2, \\
Y_2 & = X_5 - (c_{24}^4 + c_{25}^5) X_3 + (c_{25}^3 - X_2(c_{25}^4 + c_{25}^5) + c_{25}^4 (c_{14}^4 + c_{15}^5) + c_{25}^5(c_{24}^4 + c_{25}^5)) X_1 \\
& \qquad -  (c_{15}^3 - X_1(c_{25}^4 + c_{25}^5) + c_{15}^4 (c_{14}^4 + c_{15}^5) + c_{15}^5(c_{24}^4 + c_{25}^5)) X_2,
\end{align*}
and let $\bar{g}$ be the Riemannnian metric making $X_1, X_2, Z, Y_1, Y_2$ into an orthonormal basis. Write $\nabla^{\bar{g}}$ for the corresponding Levi-Civita connection.

Then $\bar{g}$, $V_{2}= \spn \{Z\}$ and $V_{3}= \spn \{ Y_1, Y_2\}$ are independent of choice of basis $X_1, X_2$. Furthermore, if we define a connection $\nabla$ making the decomposition $V = E \oplus V_{2} \oplus V_{3}$ parallel and further determined by the rules $\nabla Z = 0$ and
\begin{align*}
\langle \nabla_{X_i} X_j, X_k \rangle_{\bar{g}} & =\langle \nabla_{X_i} Y_j, Y_k \rangle_{\bar{g}} = \langle \nabla_{X_i}^{\bar{g}} X_j, X_k \rangle_{\bar{g}}, \\
\langle \nabla_{Z} X_j, X_k \rangle_{\bar{g}} & = \langle\nabla_{Z} Y_j, Y_k \rangle_{\bar{g}} = \langle [Z,X_j], X_k \rangle_{\bar{g}} + \frac{1}{2} (\mathcal{L}_Z \bar{g})(X_j, X_k), \\
\langle \nabla_{Y_i} X_j, X_k \rangle_{\bar{g}} & =\langle \nabla_{Y_i} Y_j, Y_k \rangle_{\bar{g}} = \langle [Y_i,X_j], X_k \rangle_{\bar{g}} + \frac{1}{2} (\mathcal{L}_{Y_i} \bar{g})(X_j, X_k).
\end{align*}
then $(M,E,g)$ is locally isometric to the Carnot group with growth vector $(2,3,5)$ if and only if the curvature vanishes and the only non-zero parts of the torsion $T$ are given by
$$T(X_2, X_1) = Z, \qquad T(Z, X_1) = Y_1, \qquad T(Z, X_2) = Y_2.$$
\end{theorem}

\bibliographystyle{alpha}
\bibliography{ErlendBib}

\Addresses

\end{document}